\documentclass[9pt]{article}
\usepackage{listings}
\usepackage{pdflscape}
\usepackage{amsfonts}
\usepackage{epsfig}
\usepackage{lscape}
\usepackage{longtable}
\usepackage{amsmath,amssymb,amsthm}
\usepackage{graphicx}
\usepackage{subfigure}
\usepackage{color}
\usepackage{placeins}
\usepackage{url}
\usepackage{cases}
\usepackage{longtable}
\usepackage{supertabular}
\usepackage{multirow}

\usepackage{amsmath}
\allowdisplaybreaks[4]
\usepackage{cite}
\usepackage{algorithmic}
\usepackage{algorithm,float}
\usepackage{booktabs}
\usepackage[shortlabels]{enumitem}
\usepackage[online]{threeparttablex}
\usepackage{siunitx}
\usepackage{tikz}
\usepackage{bm}
\def \sgn{{\rm sgn} }
 

\newtheorem{theorem}{Theorem}
\newtheorem{assumption}{Assumption}
 \newtheorem{example}{Example}

\newtheorem{remark}{Remark}
\newtheorem{definition}{Definition}
\newtheorem{proposition}{Proposition}
\newtheorem{corollary}{Corollary}

\begin{document}
\title{\bf \Large{DC algorithms for a class of sparse group $\ell_0$ regularized optimization problems}}
\author{Wenjing Li\footnotemark[1], \quad Wei Bian \footnotemark[2], \quad Kim-Chuan Toh\footnotemark[3]}

\date{}
\maketitle

\renewcommand{\thefootnote}{\fnsymbol{footnote}}
\footnotetext[1]{Department of Mathematics, National University of Singapore, 10 Lower Kent Ridge Road, Singapore ({\tt liwenjingsx@163. com}). The research of this author was supported by the China Scholarship Council while visiting the National University of Singapore.}
\footnotetext[2]{Corresponding author. School of Mathematics, Harbin Institute of Technology, Harbin 150001, China ({\tt bianweilvse520 @163.com}). The research of this author is supported in part by the National Natural Science Foundation of China under Grant 11871178 and Grant 61773136.}
\footnotetext[3]{Department of Mathematics and Institute of Operations Research and Analytics, National University of Singapore, 10 Lower Kent Ridge Road, Singapore ({\tt mattohkc@nus.edu.sg}). The research of this author is supported by the Ministry of Education, Singapore, under its Academic Research Fund Tier 3 grant call (MOE-2019-T3-1-010).}
\renewcommand{\thefootnote}{\arabic{footnote}}

\begin{abstract}
In this paper, we consider a class of sparse group $\ell_0$ regularized optimization problems. Firstly, we give a continuous relaxation model of the considered problem and define a class of stationary points of the relaxation problem. Then, we establish the equivalence of these two problems in the sense of global minimizers, and prove that the defined stationary point is equivalent to the local minimizer of the considered sparse group $\ell_0$ regularized problem with a desirable bound from its global minimizers. Further, based on the difference-of-convex (DC) structure of the relaxation problem, we design two DC algorithms to solve the relaxation problem. We prove that any accumulation point of the iterates generated by them is a local minimizer with a desirable bound for the considered sparse group $\ell_0$ problem. In particular, all accumulation points have a common support set and their zero entries can be attained within finite iterations. Moreover, we give the global convergence analysis of the proposed algorithms. Finally, we perform some numerical experiments to show the efficiency of the proposed algorithms.
\end{abstract}

\\[5pt]

\section{Introduction}
Over the last decade, sparse optimization problems have attracted a great deal of attention in science and engineering, such as variable selection, image restoration and wireless communication. The main purpose of these problems is to seek a sparsest solution of an underdetermined linear or nonlinear system of equations. A typical example is to recover a sparse signal $x\in\mathbb{R}^n$ from observation $b\in\mathbb{R}^m$ by considering the linear system $y=Ax+\epsilon$ with sensing matrix $A\in\mathbb{R}^{m\times n}$ and observation error $\epsilon\in\mathbb{R}^m$. Define $\mathcal{I}:\mathbb{R}\rightarrow\mathbb{R}$ by $\mathcal{I}(t)=1$ if $t\neq0$ and $\mathcal{I}(t)=0$ otherwise. The sparsity of vector $x$ on $\mathbb{R}^n$ is usually provided by its $\ell_0$-norm, denoted by $\|\cdot\|_0$, and defined by \begin{equation*}
	\|x\|_0=\sum_{i=1}^{n}\mathcal{I}(x_i).
\end{equation*} A vector $x\in\mathbb{R}^n$ is said to be sparse if $\|x\|_0\ll n$. $\ell_0$-norm plays a crucial role in sparse optimization problems as it can directly penalize the number of nonzero elements and 
promote the accurate identification of important predictors \cite{Nikolova2016Relationship}. However, the discontinuity of $\ell_0$-norm makes the nonconvex optimization problems involving $\ell_0$-norm highly challenging \cite{LeThi2015}. The $\ell_0$ regularized optimization problem usually takes this form $\min_{x\in\mathbb{R}^n}f(x)+\lambda\|x\|_0$,
where $f$ is a given loss function and $\lambda$ is a hyperparameter characterizing the trade-off between 
the loss defined by $f$ and the sparsity of $x$. In view of different application backgrounds, the function $f$ has a variety of possible expressions \cite{LeThi2015}. However, such problems are well known to be NP-hard in general \cite{Natarajan1995}.

Group-sparsity is an important class of structured sparsity and is referred to as block-sparsity in compressed sensing \cite{Duarte2011}. Group-sparsity is imperative in many cases and can improve the performance of a larger family of regression problems \cite{Eldar2009,Jenatton2011,Stojnic2009}. When the data has a certain group sparse structure and its variables are also sparse, we are naturally interested in selecting both important groups and important variables in the selected groups. In order to expand the application fields of $\ell_0$ regularized optimization problems, we study the following more general $\ell_0$ sparse group optimization problems with box constraints in this paper:
\begin{equation}\label{g0p}
\min_{x\in\Omega}~F_0(x):=f(x)+\lambda_1\|x\|_0+\lambda_2\sum_{l=1}^{L}w_l\mathcal{I}(\|x_{(l)}\|_p),
\end{equation}
where $\Omega=[\underline{u},\bar{u}]$ with $-\underline{u},\bar{u}\in\overline{\mathbb{R}^n_+}$ and $\underline{u}<\bar{u}$, $f:\mathbb{R}^n\rightarrow\mathbb{R}$ is convex, $\lambda_1>0$, $\lambda_2\geq0$, $p=1$ or $2$, $w_l\geq0$ and $x_{(l)}\in\mathbb{R}^{|G_l|}$ is the restriction of $x$ to the index set $G_l\subseteq\{1,\ldots,n\}$ for $l\in\{1,\ldots,L\}$. Without loss of generality, we suppose that $\bigcup_{l=1}^LG_l=\{1,\ldots,n\}$. Note that $\|x\|_0=\sum_{l=1}^{n}\mathcal{I}(x_i)$ and $\|(\|x_{(1)}\|_p,\ldots,\|x_{(L)}\|_p)^\top\|_0=\sum_{l=1}^{L}\mathcal{I}(\|x_{(l)}\|_p)$, which is known as $\ell_{p,0}$ norm. In problem (\ref{g0p}), the setting of group-sparsity involves some possibly overlapping groups. The box constraints have also been considered in \cite{Bian2020,Chen2012,Yap2016} and shown to be beneficial to the recovery of some images than without such constraints \cite{Chen2012}. When $p=2$ and $f(x)=\|Ax-b\|^2_2$ with $A\in\mathbb{R}^{m\times n}$ and $b\in\mathbb{R}^m$, problem (\ref{g0p}) is the $\ell_0$ counterpart of the sparse group Lasso problem presented in \cite{Simon2013}, which has been widely studied and applied to different fields \cite{Eldar2009,Zhang2020MP,Zhou2017}. Problem (\ref{g0p}) with $\lambda_2=0$ has been extensively studied in \cite{Bian2020,LeThi2015,Soubies2015,Soubies2017}. When $f$ is the squared-error loss, problem (\ref{g0p}) with $p=2$ and $\Omega=\overline{\mathbb{R}^n_+}$ was applied to the diffusion magnetic resonance imaging problems in \cite{Yap2016}, where the non-monotone iterative hard thresholding algorithm was proposed to solve (\ref{g0p}) and proved to be convergent to a local minimizer of (\ref{g0p}). Moreover, when $f$ is the negative log-likelihood and $\Omega$ is convex, problem (\ref{g0p}) was studied for the estimation of multiple covariance matrices in \cite{Phan2017DC}, where two DC algorithms were proposed to solve a relaxation problem of (\ref{g0p}) with $p=1$ and $p=2$, respectively. In \cite{Pan2021}, Pan and Chen considered the constrained group sparse optimization with the objective function defined by the $\ell_{2,0}$ function $\sum_{l=1}^{L}\mathcal{I}(\|x_{(l)}\|_2)$ and established the equivalence of their penalty models to the relaxation models in the sense of global minimizers. More interesting results on group $\ell_0$ regularized or constrained optimization problems can be found in \cite{Beck2019MP,Phan2019DC,Phan2017cof}.

There exist some relaxation methods proposed to solve the $\ell_0$ regularized models. $\ell_1$-norm is the most commonly used convex relaxation of $\ell_0$-norm. Due to the wide variety of algorithms available for convex optimization problems, $\ell_1$ regularized problems have been studied in many applications \cite{Lv2013,Wright2009}. With the advent of the Big Data era, some highly efficient algorithms were proposed for solving large-scale $\ell_1$ regularized problems \cite{Li2018SP,Lin2019SP,Zhang2020MP}. On the other hand, $\ell_1$-norm often leads to over-relaxation and produces a biased estimator \cite{Candes2008,Fan2001}. For further improvement, some researchers designed continuous but nonconvex relaxations for the $\ell_0$-norm, such as the $\ell_p$-norm ($0<p<1$) \cite{Foucart2009}, capped-$\ell_1$ penalty \cite{Peleg2008}, smoothly clipped absolute deviation (SCAD) penalty \cite{Fan2001}, minimax concave penalty (MCP) \cite{Zhang2010}, continuous exact $\ell_0$ (CEL0) penalty \cite{Soubies2015}. Most of these relaxations can be recast into Difference of Convex (DC) functions \cite{Ahn2017}, which refers to the difference of two convex functions. These DC structured relaxation problems belong to DC minimization (which refers to the problem of minimizing DC functions). For the $\ell_0$ regularization, the capped-$\ell_1$ relaxation was considered in \cite{Bian2020,LeThi2015} and has been shown to be the tightest DC relaxation for the $\ell_0$-norm \cite{LeThi2015}. Accordingly, we use the capped-$\ell_1$ function $\theta(t)=\min\big\{{|t|}/{\nu},1\big\}$ with a given $\nu>0$ to relax $\mathcal{I}(t)$. Define
\begin{equation*}
\bar{f}(x):=f(x)+\frac{\lambda_1}{\nu}\|x\|_1+\frac{\lambda_2}{\nu}\sum_{l=1}^{L}w_l\|x_{(l)}\|_p~\mbox{and}~\theta(t)=|t|/\nu-\underset{i=1,2,3}\max\{\theta_i(t)\},
\end{equation*}where $\theta_1(t)=0$, $\theta_2(t)=t/\nu-1$, $\theta_3(t)=-t/\nu-1$. Then, we consider the following DC minimization as the relaxation of problem (\ref{g0p}):
\begin{equation}\label{g0ps}
\begin{split}
\min_{x\in\Omega}~F(x):=\bar{f}(x)-\lambda_1\sum_{j=1}^{n}\underset{i=1,2,3}\max\{\theta_i(x_j)\}-\lambda_2\sum_{l=1}^{L}w_l~\underset{i=1,2}\max\{\theta_i(\|x_{(l)}\|_p)\}.
\end{split}
\end{equation}
For the (group) $\ell_0$ regularized optimization problems, there exist some equivalent DC relaxation models in the sense of global minimizers \cite{Bian2020,LeThi2015,Phan2019DC,Soubies2015}. When $\lambda_2=0$, Bian and Chen \cite{Bian2020} proved the equivalence between a class of strong local minimizers of (\ref{g0p}) with box constraints and weak d-stationary points (lifted stationary points) of (\ref{g0ps}) defined based on \cite{Pang2017}. In this paper, we will consider problem (\ref{g0p}) with $\lambda_2\neq0$, which has interesting applications in signal processing, gene expression and analysis, and neuroimaging. On the other hand, for problem (\ref{g0p}) with $w_l=1$ and $G_{i}\cap G_{j}=\emptyset,1\leq i<j\leq L$, the capped-$\ell_1$ relaxation model (\ref{g0ps}) was also considered in \cite{Phan2017DC}, but the equivalence to (\ref{g0p}) has not been proved. So the another purpose of this paper is to show the equivalence between problem (\ref{g0p}) and its relaxation problem (\ref{g0ps}) in the sense of global minimizers.

A natural way to solve a DC minimization problem is by using DC algorithm. DC algorithms have been studied extensively for more than three decades \cite{LeThi2018MP}. For the general DC minimization $\min_{x\in\mathbb{R}^n}h(x)-g(x)$ with convex functions $h$ and $g$, the classical DC algorithm generates the next iterate by solving the convex optimization problem $x^{k+1}\in{\rm{argmin}}_{x\in\mathbb{R}^n}\{h(x)-\langle v^k,x\rangle\}$ for some $v^k\in\partial g(x^k)$. Recently, DC algorithm has been further developed for improving the quality of solutions and the rate of convergence \cite{Lu2019,Lu2019MP,Pang2017,Wen2018}. Most existing DC algorithms were proved to be subsequential convergent to a critical point of DC minimization for the case that $g$ is nonsmooth \cite{Francisco2020,Gotoh2018,Wen2018}. When the subtracted function is defined by $g=\max_{1\leq i\leq I}\psi_i(x)$ with convex and continuously differentiable $\psi_i$, by exploiting the structure of the subtracted function in DC minimization, Pang, Razaviyayn, and Alvarado \cite{Pang2017} proposed an enhanced DC algorithm to solve DC minimization with subsequential convergence to a d-stationary point of the considered problem. Further, Lu, Zhou and Sun \cite{Lu2019,Lu2019MP} combined the enhanced DC algorithm with some possible accelerated methods to design some DC algorithms with subsequential convergence to the d-stationary points. In problem (\ref{g0ps}), the subtracted part is the maximum of some convex functions. But based on the nondifferentiablity of $\ell_1$-norm and $\ell_2$-norm, the DC algorithms in \cite{Lu2019,Lu2019MP,Pang2017} cannot be used to solve the problem (\ref{g0ps}) directly. 
Considering the special structure of the subtracted function in (\ref{g0ps}) and inspired by the ideas in \cite{Lu2019,Lu2019MP,Pang2017}, we will explore the structure of the relaxation function to design two DC algorithms and get a stationary point satisfying a stronger optimality condition than the critical points for problem (\ref{g0ps}). In addition, though some DC algorithms are proposed to solve the relaxation problems of group $\ell_0$ regularized problems \cite{Phan2019DC,Phan2017DC}, the relationship between the proposed DC algorithms and local minimizers of the group $\ell_0$ regularized problems is not rigorously explained. In this paper, we will analyze this relationship.

In terms of global convergence analysis, there exists only a few theoretical results for the $\ell_0$ regularized optimization problems \cite{Attouch2013,Bian2020,Lu2014,Zhou2020}. Recently, for problem (\ref{g0p}) with $\lambda_2=0$, Bian and Chen \cite{Bian2020} proved the global convergence of the proposed algorithm and its convergence rate of $o(k^{-\tau})$ with $\tau\in(0,\frac{1}{2})$ on the objective function values. Then, Zhou, Pan and Xiu \cite{Zhou2020} developed a Newton-type method with global and quadratic convergence when $f$ is twice continuously differentiable and locally strongly convex around an accumulation point. Moreover, the global convergence analysis of DC algorithms mainly relies on the KL assumptions \cite{Liu2018DC,Lu2019MP,Wen2018}. Therefore, another main purpose of this paper is to propose two algorithms with global convergence and a faster convergence rate for problem (\ref{g0p}) under a proper KL assumption. In particular, we will show that the proposed KL assumption is naturally satisfied for some common loss functions in sparse regression problems.

To sum up, the main contributions of this paper are as follows.
\begin{itemize}
\item Define a class of stationary points for relaxation problem (\ref{g0ps}), which satisfies a stronger optimality condition than weak d-stationary points and critical points of (\ref{g0ps}). For problem (\ref{g0ps}) with $\lambda_2=0$, the defined stationary point is equivalent to the weak d-stationary point in \cite{Pang2017}.
\item Generalize the definition of strong local minimizer in \cite{Bian2020} from problem (\ref{g0p}) with $\lambda_2=0$ to problem (\ref{g0p}), which satisfies a desirable property of its global minimizers. Prove the equivalence between problem (\ref{g0p}) and its relaxation problem (\ref{g0ps}) in the sense of global minimizers, and the equivalence between the defined stationary point of (\ref{g0ps}) and the strong local minimizer of (\ref{g0p}) under a weaker restriction on $\nu$ than that in \cite{Bian2020}.
\item Design two DC algorithms to obtain the strong local minimizers of (\ref{g0p}). Prove that all accumulation points of the iterates generated by the proposed algorithms have a common support set and a unified lower bound for the nonzero entries, and their zero entries can be attained within finite iterations.
\item Prove the global convergence and convergence rate of the iterates generated by the proposed algorithms under some proper assumption on the loss function, where the R-linear convergence is appropriate for (\ref{g0p}) with some common loss functions in linear, logistic and Poisson regression.
\end{itemize}

We organize the remaining part of this paper as follows. In Section \ref{section2}, we define a stationary point of relaxation problem (\ref{g0ps}) and a strong local minimizer of primal problem (\ref{g0p}), and analyze their equivalence. In Section \ref{section3}, we propose two DC algorithms to solve problem (\ref{g0ps}) and prove that all accumulation points of the proposed algorithms are strong local minimizers of problem (\ref{g0p}), have a common support set and finite iteration convergence on zero entries. In Section \ref{section4}, we analyze the global convergence and convergence rates of the proposed algorithms. Finally, some numerical examples are given in Section \ref{section6} to verify the theoretical results and show the good performance of the proposed algorithms in solving problem (\ref{g0p}).

\textbf{Notations:} $\overline{\mathbb{R}^n_+}:=[0,\infty]^n$. For $d\in\{1,2,\ldots\}$, $[d]:=\{1,\ldots,d\}$. For $x\in\mathbb{R}^n$, $\|x\|:=\|x\|_2$ denotes the Euclidean norm, $\mathcal{A}(x):=\{j\in[n]:x_j=0\}$, $\mathcal{A}^c(x):=\{j\in[n]:x_j\neq0\}$, ${B}_{\delta}(x)$ denotes the open ball in $\mathbb{R}^n$ centered at $x$ with radius $\delta>0$. Denote $\textbf{k}=(k,k,\ldots,k)\in\mathbb{R}^n$. For $x\in\mathbb{R}^n$ and $\Pi\subseteq[n]$, $|\Pi|$ denotes the number of elements in $\Pi$,
$x_{\Pi}:=(x_{i_1},x_{i_2},\ldots,x_{i_{|\Pi|}})$ with $i_1,i_2,\ldots,i_{|\Pi|}\in\Pi$ and $i_1<i_2<\ldots<i_{|\Pi|}$. For $S\subseteq\mathbb{R}^n$, $\delta_S(x)=0$ if $x\in S$ and $\delta_S(x)=\infty$ otherwise. For any $t\in\mathbb{R}$, $\lfloor t\rfloor_+:=\max\{t,0\}$ and $\lfloor t\rfloor$ denotes the largest nonnegative integer not exceeding $t$ when $t\geq0$. For a convex function $h:\mathbb{R}^n\rightarrow(-\infty,\infty]$, ${\rm{dom}}h=\{x\in\mathbb{R}^n:h(x)<\infty\}$, $\partial h$ denotes the subdifferential of $h$ \cite{Rockafellar1997}, and the proximal operator of $h$, denoted by ${\rm{prox}}_h$, is the mapping from $\mathbb{R}^n$ to $\mathbb{R}^n$ defined by ${\rm{prox}}_h(z)={\rm{argmin}}_{x\in\mathbb{R}^n}\{h(x)+\frac{1}{2}\|x-z\|^2\}$.
\section{Relationships between (\ref{g0p}) and (\ref{g0ps})}\label{section2}
In this section, we firstly define a class of stationary points of the relaxation problem (\ref{g0ps}) and analyze its lower bound property for the nonzero entries. Based on this property, we prove the equivalence between problem (\ref{g0p}) and its relaxation problem (\ref{g0ps}) in the sense of global minimizers and optimal values. Then, based on this equivalence, we define a class of strong local minimizers for problem (\ref{g0p}). Finally, we prove the equivalence between the defined stationary point of (\ref{g0ps}) and the strong local minimizer of (\ref{g0p}).

In order to define a beneficial stationary point of problem (\ref{g0ps}) and build up the relationships between problems (\ref{g0p}) and (\ref{g0ps}), we make the following assumption throughout this paper.
\begin{assumption}\label{gsassum}
(i) $f$ is global Lipschitz continuous on $\Omega$.\\
(ii) $\nu$ in (\ref{g0ps}) satisfies $\nu<\min\{\frac{\lambda_1}{L_f},\vartheta\}$, where $L_f$ is a constant satisfying $L_f\geq\sup\{|\xi_j|:\xi\in\partial f(x), x\in\Omega,j\in[n]\}$ and $\vartheta:=\min\{|\underline{u}_i|,\bar{u}_j:\underline{u}_i\neq0,\bar{u}_j\neq 0,i,j\in[n]\}$.
\end{assumption}

Note that $\nu$ in Assumption \ref{gsassum} is not affected by $\lambda_2$ in (\ref{g0p}). Moreover, $L_f$ in Assumption \ref{gsassum} can be less than the Lipschitz modulus of $f$ on $\Omega$, which implies a weaker restriction on $\nu$ than Assumption 2 in \cite{Bian2020} for (\ref{g0p}) with $\lambda_2=0$.

Firstly, we define a lower bound property for $x\in\mathbb{R}^n$.
\begin{definition}\label{lowerbd}
A vector $x\in\mathbb{R}^n$ is said to have the $\nu$-lower bound property if one has $|x_j|\geq \nu$ when $x_j\neq 0$, for any $j\in[n]$.
\end{definition}

Recall the definitions of three types of stationary points, as mentioned in \cite{Pang2017},  for the DC relaxation problem (\ref{g0ps}). In order to properly formulate the stationary points of problem (\ref{g0ps}), we define some necessary notations.

For $x\in\mathbb{R}^n$, define 
\begin{eqnarray*}
\bar{I}(x)&=&\{I\in\{1,2,3\}^n:{I_j}\in{{{\rm{arg}}\max}_{i\in\{1,2,3\}}}\{\theta_i(x_j)\},j\in[n]\}, 
\\[0pt]
\bar{J}(x)&=&\{J\in\{1,2\}^L:{J_l}\in{{{\rm{arg}}\max}_{i\in\{1,2\}}}\{\theta_i(\|x_{(l)}\|_p)\},l\in[L]\},
\\[0pt]
\Theta_{I,J}(x)&=&\lambda_1\sum_{j=1}^{n}\theta_{I_j}(x_j)+\lambda_2\sum_{l=1}^{L}w_l\theta_{J_l}(\|x_{(l)}\|_p) 
\quad\mbox{with}\quad I\in\bar{I}(x), J\in\bar{J}(x).
\end{eqnarray*} 
Combining with the above notations, we show the following definitions.
\begin{enumerate}
\item[(a)] $x\in\Omega$ is called a critical point of problem (\ref{g0ps}), if $x$ satisfies that
\begin{equation*}\label{critstat}
\begin{split}
\bm{0}\in&\partial\bar{f}(x)-\partial\Big(\lambda_1\sum_{j=1}^{n}\underset{i=1,2,3}\max\theta_i(x_j)+\lambda_2\sum_{l=1}^{L}w_l\underset{i=1,2}\max\theta_i(\|x_{(l)}\|_p)\Big)+N_{\Omega}(x).
\end{split}
\end{equation*}
	
\item[(b)] $x\in\Omega$ is called a weak d(directional)-stationary point of problem (\ref{g0ps}), if there exist $I\in\bar{I}(x)$ and $J\in\bar{J}(x)$ such that $\partial\Theta_{I,J}(x)\subseteq\partial\bar{f}(x)+N_{\Omega}(x)$.
	
\item[(c)] $x\in\Omega$ is called a d-stationary point of problem (\ref{g0ps}), if $x$ satisfies
\begin{equation*}\label{dstat}
\partial\Theta_{I,J}(x)\subseteq\partial\bar{f}(x)+N_{\Omega}(x),\;\;\forall I\in\bar{I}(x), \forall J\in\bar{J}(x).
\end{equation*}
\end{enumerate}

Next, we define the special index vectors in $\bar{I}(x)$ and $\bar{J}(x)$ as follows.
\begin{eqnarray*}
&I^x&~\mbox{is the vector in}\;\; \{1,2,3\}^n\mbox{ satisfying }I^x_j=\max\{{{{\rm{arg}}\max}_{i\in\{1,2,3\}}}\{\theta_i(x_j)\}\big\},j\in[n],
\\[0pt]
&J^x&~\mbox{is the vector in}\;\; \{1,2\}^L\mbox{ satisfying }J^x_l=\max\big\{{{{\rm{arg}}\max}_{i\in\{1,2\}}}\{\theta_i(\|x_{(l)}\|_p)\}\big\},l\in[L].
\end{eqnarray*}
Note that $I^x$ and $J^x$ in $\Theta_{I^x,J^x}$ are regarded as fixed index vectors, not as variables, and they correspond to element and group sparsities, respectively. In order to keep the derivative information of $\mathcal{I}$ at $\nu$ and $-\nu$, we choose the outer pieces of $\theta$, as the definitions of $I^x$ and $J^x$. The inner pieces of $\theta$ are also applicable to the theoretical analysis of this section, provided that $|x_j|\geq\nu$ in Definition \ref{lowerbd} is replaced by $|x_j|>\nu$.

Combining the above definitions and the effect of problem (\ref{g0ps}) for solving (\ref{g0p}), we define a class of stationary points for problem (\ref{g0ps}).
\begin{definition}\label{gssdef}
We call $x\in\Omega$ an sw(strong weak)-d-stationary point of problem (\ref{g0ps}), if $x$ satisfies the property that
\begin{equation}\label{swdstat}
\lambda_1\sum_{j=1}^{n}\nabla\theta_{I^x_j}(x_j)+\lambda_2\sum_{l=1}^{L}w_l\partial\theta_{J^x_l}(\|x_{(l)}\|_p)\subseteq\partial\bar{f}(x)+N_{\Omega}(x).
\end{equation}
\end{definition}

Let $\mathcal{S}_c$, $\mathcal{S}_d$, $\mathcal{S}_{wd}$ and $\mathcal{S}_{swd}$ denote the sets of critical points, d-stationary points, weak d-stationary points, and the defined sw-d-stationary points of problem (\ref{g0ps}), respectively. We can easily find that 
they satisfy the following inclusions \begin{equation*}
\mathcal{S}_d\subseteq\mathcal{S}_{swd}\subseteq\mathcal{S}_{wd}\subseteq\mathcal{S}_c.
\end{equation*}It is known that any local minimizer of problem (\ref{g0ps}) is a d-stationary point of it, which is of course an sw-d-stationary point of it. By \cite[Proposition 2.2]{Bian2020}, the sw-d-stationary point in Definition \ref{gssdef} is equivalent to the weak d-stationary point for (\ref{g0ps}) with $\lambda_2=0$. However, the weak d-stationary point is not sufficient to be an sw-d-stationary point of (\ref{g0ps}) when $\lambda_2>0$. For example, $(\frac{1}{4},0)$ is a weak d-stationary point of problem (\ref{g0ps}) with $\nu=\frac{1}{4}$ for the problem $\min_{x:=(x_1,x_2)^\top\in[0,1]^2}|x_1-\frac{1}{4}|-\frac{3}{2}x_1+x_2^2+\|x\|_0+\frac{1}{8}\mathcal{I}(\|x\|_p),p=1,2$, but it is not an sw-d-stationary point of it. Moreover, any sw-d-stationary point $x$ of (\ref{g0ps}) is a d-stationary point of (\ref{g0ps}) if $\|x_{(l)}\|_p\neq\nu,\forall l\in[L]$ when $\lambda_2>0$, and $|x_j|\neq\nu,\forall j\in[n]$.

In what follows, we build up the relationships of problem (\ref{g0p}) and (\ref{g0ps}), where the first step is to prove $\nu$-lower bound property for sw-d-stationary points of (\ref{g0ps}).
\begin{proposition}\label{gslbdp}
Any sw-d-stationary point of (\ref{g0ps}) has $\nu$-lower bound property, and for any $l\in[L]$, $\|x_{(l)}\|_p\geq \nu$ if $x_{(l)}\neq \bm{0}$.
\end{proposition}
\begin{proof}
Let $x\in\Omega$ be an arbitrary sw-d-stationary point of (\ref{g0ps}). If the statement in this proposition does not hold, then there exists a $j\in [n]$ such that $|x_{j}|\in(0,\nu)$. It follows from the definition of $\theta$ and Assumption \ref{gsassum} that $I^x_{j}={1}$ and $N_{\Omega_{j}}(x_{j})=\{0\}$. This implies that
\begin{equation}\label{gslbd}
0\in[\partial f(x)]_{j}+\frac{\lambda_1}{\nu}\sgn(x_{j})+\lambda_2\sum_{{l}\in\bar{\mathcal{L}}_{j}}\frac{ w_{{l}}}{\nu}\nabla_{x_{j}}\|x_{({l})}\|_p,
\end{equation}
where $\bar{\mathcal{L}}_{j}:=\{l: j\in G_{l},\|x_{(l)}\|_p<\nu\}$. Since $\big|\nabla_{x_{j}}|x_{j}|\big|=1$ and $\sum_{{l}\in\bar{\mathcal{L}}_{j}}\frac{ w_{{l}}}{\nu}\nabla_{x_{j}}\|x_{({l})}\|_p$ has the same sign as $x_{j}$, by Assumption \ref{gsassum}, one has that the relationship of (\ref{gslbd}) does not hold, which leads to a contradiction. Therefore, for any $j\in[n]$, one has $x_j=0$ if $|x_j|<\nu$, which means that $x$ has $\nu$-lower bound property. Then, for any $l\in[L]$, $\|x_{(l)}\|_p\geq \nu$ if $x_{(l)}\neq \bm{0}$.
\end{proof}

Based on Proposition \ref{gslbdp},  we analyze the equivalence of problem (\ref{g0p}) and problem (\ref{g0ps}) in the sense of global minimizers and optimal values. The proof idea is similar to that for Theorem 2.4 in \cite{Bian2020}.
\begin{proposition}\label{gsglb=}
The optimal solution sets and optimal values of sparse group optimization problem (\ref{g0p}) and its continuous relaxation problem (\ref{g0ps}) are same, respectively.
\end{proposition}
\begin{proof}
For a given global minimizer $\bar{x}$ of (\ref{g0p}), if $\bar{x}$ is not a global minimizer of (\ref{g0ps}), then there exists a global minimizer $\hat{x}$  of (\ref{g0ps}) such that 
\begin{equation}\label{gsglobal}
F(\hat{x})< F(\bar{x})\leq F_0(\bar{x})\leq F_0(\hat{x}),
\end{equation}
where the second inequality follows from the fact that $\theta(t) \leq \mathcal{I}(t),\forall t\geq 0$. Since $\hat{x}$ is also an sw-d-stationary point of \eqref{g0ps}, by Proposition \ref{gslbdp}, we have that $F_0(\hat{x})=F(\hat{x})$, which leads to a contradiction to \eqref{gsglobal}.
	
Suppose $\bar{x}$ is a given global minimizer of (\ref{g0ps}). Then $\bar{x}$ is certainly an sw-d-stationary point of (\ref{g0ps}). By Proposition \ref{gslbdp}, one has that $F_0(\bar{x})=F(\bar{x})\leq F({x})\leq F_0({x}),\forall x\in\Omega$, where the last inequality follows again from the fact that $\theta(t) \leq\mathcal{I}(t),\forall t\geq 0$. Thus we have shown that $\bar{x}$ must be a global minimizer of (\ref{g0p}).
	
By the definition of $\theta$ and Proposition \ref{gslbdp}, one has that optimal values of (\ref{g0p}) and (\ref{g0ps}) are same.
\end{proof}

Since any global minimizer of (\ref{g0ps}) is an sw-d-stationary point of (\ref{g0ps}), one has that any global minimizer of (\ref{g0ps}) owns $\nu$-lower bound property by Proposition \ref{gslbdp}. Further, it follows from Proposition \ref{gsglb=} that any global minimizer of (\ref{g0p}) has $\nu$-lower bound property. Then, we have the following necessary condition for the global minimizers of problem (\ref{g0p}).
\begin{corollary}\label{gl-p}
Any global minimizer of (\ref{g0p}) has $\nu$-lower bound property.
\end{corollary}

Next, we give an example to show that the $\nu$-lower bound property is not a necessary condition for the local minimizers of (\ref{g0p}).
\begin{example}\label{g0p-exp1}
Consider the specific problem of (\ref{g0p}) as follows,
\begin{equation}\label{g0p-exp}
\min_{x:=(x_1,x_2)^\top\in[0,1]^2}~\big|x_1-x_2-\frac{1}{2}\big|+\|x\|_0+\mathcal{I}(\|x\|_p),~p=1,2.
\end{equation}
For (\ref{g0p-exp}), $\nu$ in Assumption \ref{gsassum} can be chosen in $(0,1)$. Denote the set of global minimizers, local minimizers and sw-d-stationary points by $\mathcal{GM},\mathcal{LM}$ and $\mathcal{SS}$, respectively. It follows that $\mathcal{GM}=\{(0,0)^\top\}$, $\mathcal{LM}=\{x\in[0,1]^2:x_1-x_2=\frac{1}{2}\}\cup\{(0,0)^\top\}$ and $\mathcal{SS}=\{x\in[\nu,1]^2:x_1-x_2=\frac{1}{2}\}\cup\{(x_1,0)^\top:x_1=0~\mbox{or}~\frac{1}{2}~\mbox{if}~\nu\leq\frac{1}{2}~\mbox{and}~x_1=0~\mbox{if}~\nu>\frac{1}{2}\}$. It can be seen that some local minimizers of (\ref{g0p-exp}) do not have $\nu$-lower bound property, and $\mathcal{GM}\subseteq\mathcal{SS}\subsetneq\mathcal{LM}$. In particular, if $\nu\in(\frac{1}{2},1)$, we have $\mathcal{GM}=\mathcal{SS}$ for (\ref{g0p-exp}).
\end{example}

Since the $\nu$-lower bound property is a necessary condition for global minimizers of (\ref{g0p}), but not for its local minimizers, we define the following $\nu$-strong local minimizers of (\ref{g0p}).
\begin{definition}\label{nustronglocal}
We call $x$ a $\nu$-strong local minimizer of problem (\ref{g0p}), if $x$ is a local minimizer of problem (\ref{g0p}) and owns $\nu$-lower bound property.
\end{definition}

Based on the $\nu$-lower bound property in Definition \ref{nustronglocal}, the set of $\nu$-strong local minimizers of (\ref{g0p}) gets smaller as the parameter $\nu$ gets bigger. Considering (\ref{g0p}) with $\lambda_2=0$, the weak restriction on $\nu$ in Assumption \ref{gsassum} than \cite[Assumption 2]{Bian2020}, may cause that the defined $\nu$-strong local minimizer is sufficient but not necessary to be a strong local minimizer in \cite{Bian2020}.

Next, we derive a necessary and sufficient condition for the local minimizers of (\ref{g0p}).
\begin{proposition}\label{g0plocalcondition}
$\bar{x}\in\Omega$ is a local minimizer of problem (\ref{g0p}) if and only if $\bar{x}$ is a global minimizer of $f$ on $\bar\Omega:=\{x\in\Omega:x_j=0,\forall j\in\mathcal{A}(\bar{x})\}$, which is equivalent to
\begin{equation}\label{2exp}
\bm{0}\in [\partial f(\bar{x})+N_{\Omega}(\bar{x})]_{\mathcal{A}^c(\bar{x})}.
\end{equation}
\end{proposition}
\begin{proof}
For a given vector $\bar{x}\in\mathbb{R}^n$, there exists a $\delta>0$ such that \begin{equation}\label{AC-rela}
\mathcal{A}^c(\bar{x})\subseteq\mathcal{A}^c(x),\forall x\in B_{\delta}(\bar{x}),
\end{equation}
which implies
\begin{equation}\label{im==00}
\|\bar{x}\|_0=\|x\|_0~{\rm{and}}~\sum_{l=1}^{L}w_l\mathcal{I}(\|\bar{x}_{(l)}\|_p)=\sum_{l=1}^{L}w_l\mathcal{I}(\|{x}_{(l)}\|_p),\forall x\in\bar\Omega\cap B_{\delta}(\bar{x}).
\end{equation}
(i) Let $\bar{x}$ be a local minimizer of (\ref{g0p}). Then there exists a $\delta'\in(0,\delta]$ such that $F_0(\bar{x})\leq F_0(x),\forall x\in B_{\delta'}(\bar{x})$. By (\ref{im==00}), we have $f(\bar{x})\leq f({x}),\forall x\in \bar\Omega\cap B_{\delta'}(\bar{x}).$ Then $\bar{x}$ is a local minimizer of $f$ on $\bar{\Omega}$. Since $f$ is convex, we have that $\bar{x}$ is a global minimizer of $f$ on $\bar{\Omega}$.\\
(ii) If $\bar{x}$ is a global minimizer of $f$ on $\bar{\Omega}$, then $f(\bar{x})\leq f({x}),\forall x\in \bar\Omega$. In this case, we also have (\ref{im==00}). Then, we deduce that \begin{equation}\label{mequ1}
F_0(\bar{x})\leq F_0(x),\forall x\in \bar\Omega\cap B_{\delta}(\bar{x}).
\end{equation}On the other hand, by the continuity of $f$, there exists a $\delta''\in(0,\delta]$ such that $f(\bar{x})\leq f({x})+\lambda_1,\forall x\in \Omega\cap B_{\delta''}(\bar{x})$. For any $x\in(\Omega\setminus \bar{\Omega})\cap B_{\delta''}(\bar{x})$, we have $\|\bar{x}\|_0+1\leq\|x\|_0$ by (\ref{AC-rela}) and hence $\lambda_1\|\bar{x}\|_0+\lambda_2\sum_{l=1}^{L}w_l\mathcal{I}(\|\bar{x}_{(l)}\|_p)+\lambda_1\leq\lambda_1\|x\|_0+\lambda_2\sum_{l=1}^{L}w_l\mathcal{I}(\|{x}_{(l)}\|_p)$. Then, we have
\begin{equation}\label{mequ2}
F_0(\bar{x})\leq f(x)+\lambda_1+\lambda_1\|\bar{x}\|_0+\lambda_2\sum_{l=1}^{L}w_l\mathcal{I}(\|\bar{x}_{(l)}\|_p)\leq F_0(x),\forall x\in(\Omega\setminus \bar{\Omega})\cap B_{\delta''}(\bar{x}).
\end{equation}
Therefore, combining (\ref{mequ1}) and (\ref{mequ2}), $\bar{x}$ is a global minimizer of (\ref{g0p}) on $\Omega\cap B_{\delta''}(\bar{x})$, which means that $\bar{x}$ is a local minimizer of (\ref{g0p}).\\
(iii) By the convexity of $f$, $\bar{x}\in\Omega$ is a global minimizer of $f$ on $\bar{\Omega}$ is equivalent to that $\bm{0}\in\partial f(\bar{x})+N_{\bar{\Omega}}(\bar{x})$. Based on $[N_{\bar{\Omega}}(\bar{x})]_{\mathcal{A}(\bar{x})}=\mathbb{R}^{|\mathcal{A}(\bar{x})|}$ and $[N_{\bar{\Omega}}(\bar{x})]_{\mathcal{A}^c(\bar{x})}=[N_{\Omega}(\bar{x})]_{\mathcal{A}^c(\bar{x})}$, we obtain that $\bm{0}\in\partial f(\bar{x})+N_{\bar{\Omega}}(\bar{x})$ is equivalent to (\ref{2exp}).
\end{proof}

Further, we analyze the relationship between $\nu$-strong local minimizers of problem (\ref{g0p}) and sw-d-stationary points of relaxation problem (\ref{g0ps}).
\begin{proposition}\label{local-stat}
$\bar{x}$ is an sw-d-stationary point of (\ref{g0ps}) if and only if it is a $\nu$-strong local minimizer of (\ref{g0p}). Moreover, the corresponding objective function values are same.
\end{proposition}
\begin{proof}
First, let $\bar{x}$ be an sw-d-stationary point of (\ref{g0ps}). By Proposition \ref{gslbdp}, $\bar{x}$ has the $\nu$-lower bound property. By Definition \ref{gssdef}, the definitions of $I^x$ and $J^x$, and the $\nu$-lower bound property of $\bar{x}$, one has (\ref{2exp}), which implies $\bar{x}$ is a local minimizer of (\ref{g0p}) by Proposition \ref{g0plocalcondition}. Therefore, $\bar{x}$ is a $\nu$-strong local minimizer of (\ref{g0p}).
	
Second, let $\bar{x}$ be a $\nu$-strong local minimizer of (\ref{g0p}). By the $\nu$-lower bound property of $\bar{x}$, we have $\big\{\big[\lambda_1\sum_{j=1}^{n}\nabla\theta_{I^{\bar{x}}_j}(\bar{x}_j)\big]_{\mathcal{A}^c(\bar{x})}\big\}=\big[\frac{\lambda_1}{\nu}\partial\|\bar{x}\|_1\big]_{\mathcal{A}^c(\bar{x})}$. Adding them to both sides of (\ref{2exp}), respectively, we obtain that there exists a $\bar\zeta\in\partial{f}(\bar{x})$ such that
\begin{equation}\label{inclcase1}
\Big[\lambda_1\sum_{j=1}^{n}\nabla\theta_{I^{\bar{x}}_j}(\bar{x}_j)\Big]_{\mathcal{A}^c(\bar{x})}\in\Big[\bar\zeta+\frac{\lambda_1}{\nu}\partial\|\bar{x}\|_1+N_{\Omega}(\bar{x})\Big]_{\mathcal{A}^c(\bar{x})}.
\end{equation}
Based on Assumption \ref{gsassum} and $\bm{0}\in N_{\Omega}(\bar{x})$, we deduce that $\bm{0}\in[\bar\zeta+\frac{\lambda_1}{\nu}\partial\|\bar{x}\|_1+N_{\Omega}(\bar{x})]_{\mathcal{A}(\bar{x})}$.
Further, since $\big[\lambda_1\sum_{j=1}^{n}\nabla\theta_{I^{\bar{x}}_j}(\bar{x}_j)\big]_{\mathcal{A}(\bar{x})}=\bm{0}$, we have $\Big[\lambda_1\sum_{j=1}^{n}\nabla\theta_{I^{\bar{x}}_j}(\bar{x}_j)\Big]_{\mathcal{A}(\bar{x})}\in\Big[\bar\zeta+\frac{\lambda_1}{\nu}\partial\|\bar{x}\|_1+N_{\Omega}(\bar{x})\Big]_{\mathcal{A}(\bar{x})}$.
Then, by (\ref{inclcase1}), $\partial\|\bar{x}\|_1=[\partial\|\bar{x}\|_1]_{\mathcal{A}(\bar{x})}\times[\partial\|\bar{x}\|_1]_{\mathcal{A}^c(\bar{x})}$ and $N_{\Omega}(\bar{x})=[N_{\Omega}(\bar{x})]_{\mathcal{A}(\bar{x})}\times[N_{\Omega}(\bar{x})]_{\mathcal{A}^c(\bar{x})}$, we deduce $\lambda_1\sum_{j=1}^{n}\nabla\theta_{I^{\bar{x}}_j}(\bar{x}_j)\in\bar\zeta+\frac{\lambda_1}{\nu}\partial\|\bar{x}\|_1+N_{\Omega}(\bar{x})$. Since $\bar\zeta\in\partial{f}(\bar{x})$, we obtain
\begin{equation}\label{lllzz}
\lambda_1\sum_{j=1}^{n}\nabla\theta_{I^{\bar{x}}_j}(\bar{x}_j)\in\partial {f}(\bar{x})+\frac{\lambda_1}{\nu}\partial\|\bar{x}\|_1+N_{\Omega}(\bar{x}).
\end{equation}
Moreover, by the $\nu$-lower bound property of $\bar{x}$, we have $\lambda_2\sum_{l=1}^{L}w_l\partial\theta_{J^x_l}(\|\bar{x}_{(l)}\|_p)\subseteq\frac{\lambda_2}{\nu}\sum_{l=1}^{L}w_l\partial\|\bar{x}_{(l)}\|_p$. Adding them to both sides of (\ref{lllzz}), respectively, we obtain (\ref{swdstat}) with $x$ replaced by $\bar{x}$. Therefore, $\bar{x}$ is an sw-d-stationary point of (\ref{g0ps}).
	
Based on Proposition \ref{gslbdp}, one has that $F_0(\bar{x})=F(\bar{x})$, when $\bar{x}$ is an sw-d-stationary point of (\ref{g0ps}) or a $\nu$-strong local minimizer of (\ref{g0p}).
\end{proof}

\begin{remark}
In this section, the $\nu$-lower bound property of sw-d-stationary point of (\ref{g0ps}) and the equivalence theories are based on the following two necessary properties of the relaxation function $\theta$ for $\mathcal{I}$:
	
(i) $\theta(0)=0$ and $\theta(t)=1$ when $|t|\geq\nu$;
	
(ii) $|\theta'(t)|$ has a uniform positive lower bound on $(-\nu,0)\cup(0,\nu)$.
	
Capped-$\ell_1$ function is a simple relaxation function of $\mathcal{I}$ satisfying both properties (i) and (ii). But none of the relaxation functions $\ell_p(0<p\leq 1)$, SCAD and MCP satisfy both properties (i) and (ii) at the same time. Some equivalences in the sense of global minimizers between (\ref{g0p}) with $\lambda_2=0$ and its relaxation problems can be found in \cite{Bian2020,LeThi2015,Soubies2017} and the references therein.
\end{remark}

In conclusion, the relationships between sparse group $\ell_0$ regularized problem (\ref{g0p}) and its DC relaxation problem (\ref{g0ps}) under Assumption \ref{gsassum} are shown in Fig.\ref{figure0}. The ``property'' in Fig.\ref{figure0} means that point $x$ satisfies $\|x_{(l)}\|_p\neq\nu,\forall l\in[L]$ if $\lambda_2>0$, and $|x_j|\neq\nu,\forall j\in[n]$.
\begin{figure}
\centering
\includegraphics[width=6in]{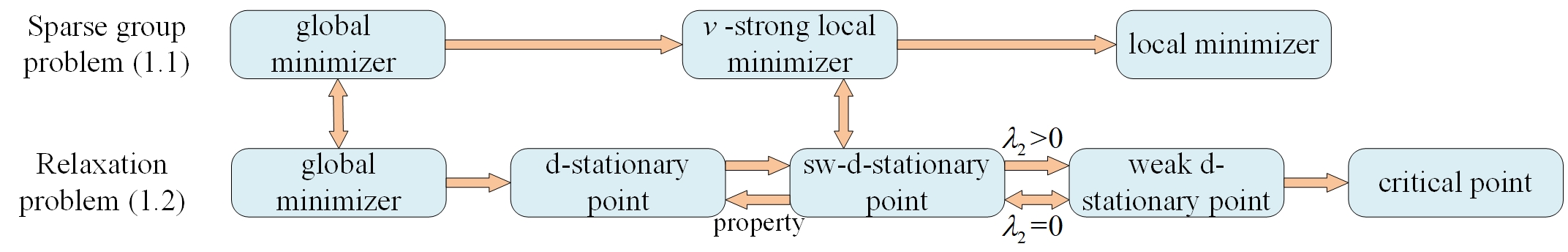}
\caption{Relationships between sparse group problem (\ref{g0p}) and relaxation problem (\ref{g0ps})}\label{figure0}
\end{figure}
\section{DC Algorithms}\label{section3}
In this section, according to the formulation (\ref{swdstat}) in Definition \ref{gssdef}, we construct two kinds of DC algorithms to solve the DC relaxation problem (\ref{g0ps}). Different from the existing DC algorithms for finding critical points, we prove that any accumulation point of the proposed DC algorithms is an sw-d-stationary point of problem (\ref{g0ps}), which is a $\nu$-strong local minimizer of problem (\ref{g0p}). In particular, for the proposed algorithms, we prove that all accumulation points of the generated iterates have the same location of zero entries and the corresponding entries of the generated iterates converge to zero within finite iterations.

First, we make the following assumption throughout the subsequent content.
\begin{assumption}\label{as-h}
(i) $f$ is level-bounded on $\Omega$.\\
(ii) $f:=f_n+f_s$, where $f_n$ is convex but not necessarily differentiable, $f_s$ is differentiable convex and its gradient is Lipschitz continuous with Lipschitz modulus $L_s>0$.\\
(iii) $L_f$ in Assumption \ref{gsassum} is not less than $\sup\{|\xi_j|:\xi\in\partial f_n(x)+\partial f_s(y),x,y\in\Omega,j\in[n]\}$.\\
(iv) For any $\mu>0$, the proximal operator of $f_n(x)+\frac{\lambda_1}{\mu}\|x\|_1+\frac{\lambda_2}{\mu}\sum_{l=1}^{L}w_l\|x_{(l)}\|_p+\delta_\Omega(x)$ can be evaluated.
\end{assumption}

There are many examples satisfying Assumption \ref{as-h}, for example, $f_n(x)=\|x\|_1$, $f_s(x)=\|Ax-b\|^2$ with $A\in\mathbb{R}^{m\times n}$ and $b\in\mathbb{R}^m$, $\Omega=[0,1]^n$ and $p=1$. More applicable examples and calculation of proximal operators can be found in Remark \ref{alg1-sp}.

Next, we give the following necessary definitions throughout this paper.

Given $K\in\mathbb{N}$ and sequence $\{\bar{\mu}_k\}$ with $\bar{\mu}_k>0,k=0,1,\ldots,K-1$, define
\begin{equation}\label{mukexp}
\mu_k=\left\{
\begin{aligned}
&\bar\mu_k,&&k=0,1,\ldots,K-1,\\
&\nu,&&k=K,K+1\ldots.
\end{aligned}
\right.
\end{equation}For $\mu>0$ and $t\in\mathbb{R}$, let $\theta_1(t;\mu)=0$, $\theta_2(t;\mu)=t/\mu-1$, $\theta_3(t;\mu)=-t/\mu-1$. Clearly, $\theta_i(t;\nu)=\theta_i(t),i=1,2,3$. For $I\in\{1,2,3\}^n$, $J\in\{1,2\}^L$, $x\in\mathbb{R}^n$ and $\mu>0$, define $\Theta_{I,J}(x;\mu):=\lambda_1\sum_{j=1}^n\theta_{I_j}(x_j;\mu)+\lambda_2\sum_{l=1}^{L}w_l\theta_{J_l}(\|x_{(l)}\|_p;\mu)$. Clearly, $\Theta_{I,J}(x;\nu)=\Theta_{I,J}(x)$. For $I\in\{1,2,3\}^n$, $J\in\{1,2\}^L$, $x,y,z\in\mathbb{R}^n$, $\mu>0$ and $\xi\in\partial_z \Theta_{I,J}(z;\mu)$, define $\bar{f}_n(x;\mu):=f_n(x)+\frac{\lambda_1}{\mu}\|x\|_1+\frac{\lambda_2}{\mu}\sum_{l=1}^{L}w_l\|x_{(l)}\|_p$ and
\begin{equation*}
{w}_{I,J,\xi}(x;\mu,y,z):=\bar{f}_n(x;\mu)+f_s(y)+\langle\nabla f_s(y),x-y\rangle-\Theta_{I,J}(z;\mu)-\langle\xi,x-z\rangle.
\end{equation*}

Before presenting the designed algorithms, we analyze some necessary properties, which will play an important role in the convergence analysis of the proposed algorithms. Based on Proposition 2 in \cite{Lu2019}, combining the optimality condition and the upper semicontinuity for the subdifferential of convex functions, we have the following proposition.

\begin{proposition}\label{coro1}
Let $I\in\{1,2,3\}^n$ and $J\in\{1,2\}^L$, $\{y^k\}$ and $\{z^k\}$ be two sequences of vectors in $\Omega$ converging to some $y^*$ and $z^*$, respectively, $\{\xi^k\}$ be a sequence of vectors in $\partial_x \Theta_{I,J}(z^k)$ converging to some $\xi^*$ and $\{\alpha_k\}$ be a sequence of positive scalars converging to some $\alpha_*>0$. Suppose that the sequence $\{x^k\}$ is given by
\begin{equation*}\label{limit1}
x^{k}=\underset{x\in\Omega}{\rm{argmin}}\{w_{I,J,\xi^k}(x;\nu,y^k,z^k)+\frac{\alpha_k}{2}\|x-y^k\|^2\}.
\end{equation*}
Then $\{x^k\}$ converges to $\underset{x\in\Omega}{\rm{argmin}}\{w_{I,J,\xi^*}(x;\nu,y^*,z^*)+\frac{\alpha_*}{2}\|x-y^*\|^2\}$.
\end{proposition}

For the $k$-th iteration, we design the subproblems in the proposed algorithms corresponding to the following model with $\mu=\mu_k$,
\begin{equation}\label{g0psa}
\min_{x\in\Omega}~F(x;\mu):=\bar{f}(x;\mu)-\lambda_1\sum_{j=1}^{n}\underset{i=1,2,3}\max\{\theta_i(x_j;\mu)\}-\lambda_2\sum_{l=1}^{L}w_l~\underset{i=1,2}\max\{\theta_i(\|x_{(l)}\|_p;\mu)\}.
\end{equation}
Since $F(x;\nu)=F(x)$, based on the definition (\ref{mukexp}) of $\mu_k$, $F(x;\mu_k)=F(x)$ as $k\geq K$. Thus, problem (\ref{g0psa}) is just problem (\ref{g0ps}) when $k\geq K$. The good performance of the proposed algorithms with the updated $\mu_k$ instead of the fixed parameter $\nu$ will be further explained in Subsection \ref{munes}. Similar to the idea in \cite{Bian2020}, there is only one subproblem in every iteration of the proposed algorithms, which has fewer calculation than the DC algorithms in \cite{Lu2019,Lu2019MP,Pang2017}. 

Based on the above ideas, we design two kinds of specific DC algorithms to solve problem (\ref{g0ps}) with convergence to an sw-d-stationary point in Definition \ref{gssdef}.

For simplicity, we use $I^k$ and $J^k$ to denote $I^{x^k}$ and $J^{x^k}$ throughout this paper, where $\{x^k\}$ is the iterates generated by the proposed algorithms and $\theta_i(\cdot)$ in their definitions is replaced by $\theta_i(\cdot;\mu_k)$.
\subsection{A DC algorithm with line search}
In this part, we propose a DC algorithm with line search to solve problem (\ref{g0ps}). 
\begin{algorithm}
\caption{DC algorithm with line search}\label{g0palg1}
{\bf Initialization:}
Choose $x^0\in\Omega$, $\rho>1$, $c\in(0,L_s]$, $0<\underline{\alpha}\leq\bar{\alpha}$ and integer $N\geq 0$. Set $\{\mu_k\}$ be defined as in (\ref{mukexp}) and $k=0$.
	
{\bf Step 1:} Choose $\alpha_{k}^B\in[\underline{\alpha},\bar{\alpha}]$.
	
{\bf Step 2:} For $m=1,2,\ldots$:
	
\hspace*{0.5in} (2a) Let $\alpha_k=\alpha_{k}^B\rho^{m-1}$.
	
\hspace*{0.5in} (2b) Take $\xi^k\in\partial_x \Theta_{I^k,J^k}(x^k;\mu_k)$ and compute $x^{k}(\alpha_k)$ by
\begin{equation}\label{alg12}
x^{k}(\alpha_k)=\underset{x\in\Omega}{\rm{argmin}}\Big\{w_{I^k,J^k,\xi^k}(x;\mu_k,x^k,x^k)+\frac{\alpha_k}{2}\|x-x^k\|^2\Big\}.
\end{equation}
	
\hspace*{0.5in} If $x^{k}(\alpha_k)$ satisfies
\begin{equation}\label{alg11}
F(x^{k}(\alpha_k);\mu_k)\leq\max_{\lfloor k-N\rfloor_+\leq j\leq k}\{F(x^j;\mu_k)\}-\frac{c}{2}\|x^{k}(\alpha_k)-x^k\|^2,
\end{equation}
	
\hspace*{0.5in} set $x^{k+1}=x^{k}(\alpha_k)$ and $\bar{\alpha}_k=\alpha_k$. Then go to Step 3.
	
{\bf Step 3:} Update $k\leftarrow k+1$ and return to Step 1.
\end{algorithm}

For Algorithm \ref{g0palg1}, $\{F(x^k;\mu_k)\}$ is monotone when $N=0$ and is likely to be nonmonotone when $N>0$.

\begin{remark}\label{alg1-sp}
The subproblem (\ref{alg12}) in Algorithm \ref{g0palg1} is equivalent to the following form with $\zeta^k=x^k-\frac{1}{\alpha_k}\big(\nabla f_s(x^k)-\xi^k\big)$,
\begin{equation}\label{alg1-sb}
x^{k}(\alpha_k)={\rm{prox}}_{\frac{1}{\alpha_k}\big(\bar{f}_n(\cdotp;\mu_k)+\delta_\Omega(\cdot)\big)}(\zeta^k).
\end{equation}Next, we consider the calculation of (\ref{alg1-sb}) in the case of $f=f_s+\lambda\|\cdot\|_1$ with $\lambda\geq0$.
	
(i) Suppose $p=1$ and $\Omega=[l,u]$ with $-l,u\in\overline{\mathbb{R}^n_+}$. Let $\mathcal{L}_{j}:=\{l\in[L]:j\in G_{l}\}$. Then, $\bar{f}_n(x;\mu_k)=\sum_{j=1}^n\tilde\lambda_j|x_j|$ with $\tilde\lambda_j=(\lambda\mu_k+\lambda_1+\lambda_2\sum_{l\in\mathcal{L}_{j}}w_l)/\mu_k$. Based on the variable separation of $\bar{f}_n(\cdotp;\mu_k)$, ${\rm{prox}}_{\frac{1}{\alpha_k}\bar{f}_n(\cdotp;\mu_k)}(\zeta^k)$ is a variant of proximal operator of $\ell_1$ norm and $x^{k}(\alpha_k)$ in (\ref{alg1-sb}) is the projection of ${\rm{prox}}_{\frac{1}{\alpha_k}\bar{f}_n(\cdotp;\mu_k)}(\zeta^k)$ on $\Omega$.
	
(ii) If $p=2$, $\Omega=\mathbb{R}^n$ and $G_{i}\cap G_{j}=\emptyset,1\leq i<j\leq L$, we have $\bar{f}_n(x;\mu_k)=(\lambda+\frac{\lambda_1}{\mu_k})\|x\|_1+\frac{\lambda_2}{\mu_k}\sum_{l=1}^{L}w_l\|x_{(l)}\|$ and $x^{k}(\alpha_k)$ in (\ref{alg1-sb}) is ${\rm{prox}}_{\frac{1}{\alpha_k}\bar{f}_n(\cdotp;\mu_k)}(\zeta^k)$.
	
The corresponding closed-form solutions of ${\rm{prox}}_{\frac{1}{\alpha_k}\bar{f}_n(\cdotp;\mu_k)}(\zeta^k)$ can be found in \cite{Zhang2020MP}.
\end{remark}

Firstly, we show that for each outer loop, its associated inner loops must terminate in finite number of iterations.
\begin{proposition}\label{innerloops}
There exists a positive integer $M$ such that Step 2 of Algorithm \ref{g0palg1} terminates at some $\alpha_k\leq\max\{\bar{\alpha},\rho L_s\}$ in $M$ iterations for any $k\geq0$.
\end{proposition}
\begin{proof}
If (\ref{alg11}) is satisfied with $\alpha_k\geq L_s$, then by $\rho>1$, we have that $\underline\alpha\rho^{m-1}\leq\alpha_k\leq\max\{\bar{\alpha},\rho L_s\}$ and hence $m\leq\frac{\log(\max\{\bar{\alpha},\rho L_s\})-\log\underline\alpha}{\log\rho}+1:=M'$, which means that this proposition holds with $M=\left\lfloor M'\right\rfloor+1$. Next, it is only necessary to prove that (\ref{alg11}) is satisfied when $\alpha_k\geq L_s$. 
	
Suppose $\alpha_k\geq L_s$. Then, we have that
\begin{equation*}
\begin{split}
&\max_{\lfloor k-N\rfloor_+\leq j\leq k}\{F(x^j;\mu_k)\}-\frac{c}{2}\|x^{k}(\alpha_k)-x^k\|^2\\\geq& \bar{f}_n(x^k;\mu_k)+f_s(x^k)-\Theta_{I^k,J^k}(x^k;\mu_k)-\frac{\alpha_k}{2}\|x^k-x^{k}(\alpha_k)\|^2\\\geq& w_{I^k,J^k,\xi^k}(x^{k}(\alpha_k);\mu_k,x^k,x^k)+\frac{\alpha_k}{2}\|x^k-x^{k}(\alpha_k)\|^2\\=&\bar{f}_n(x^{k}(\alpha_k);\mu_k)+f_s(x^k)+\langle\nabla f_s(x^k),x^{k}(\alpha_k)-x^k\rangle+\frac{\alpha_k}{2}\|x^k-x^{k}(\alpha_k)\|^2\\&-\Theta_{I^k,J^k}(x^k;\mu_k)-\langle\xi^k,x^{k}(\alpha_k)-x^k\rangle\\\geq&\bar{f}_n(x^{k}(\alpha_k);\mu_k)+f_s(x^{k}(\alpha_k))-\Theta_{I^k,J^k}(x^{k}(\alpha_k);\mu_k)\geq F(x^{k}(\alpha_k);\mu_k),
\end{split}
\end{equation*}
where the first and last inequalities follow from the definitions of $F(x;\mu)$, $I^x$ and $J^x$, and $\alpha_k\geq L_s\geq c$, the second inequality is based on the optimality condition of (\ref{alg12}) and the strongly convexity (modulus $\alpha_k$) of objective function in (\ref{alg12}) with respect to $x$, and the third inequality is due to the Lipschitz continuity of $\nabla f_s$, $\alpha_k\geq L_s$ and the convexity of $\Theta_{I^k,J^k}(\cdot,\mu)$. Therefore, (\ref{alg11}) is satisfied when $\alpha_k\geq L_s$. 
\end{proof}

Next, we give some basic properties of Algorithm \ref{g0palg1}.
\begin{proposition}\label{alg1prop}
Let $\{x^k\}$ be the iterates generated by Algorithm \ref{g0palg1}. Then we have that $\{x^k\}$ is bounded, $\lim_{k\rightarrow\infty}\|x^{k+1}-x^k\|=0$ and $\lim_{k\rightarrow\infty}F(x^k)$ exists.
\end{proposition}
\begin{proof}
For any given $\mu>0$, define $\iota(k)$ as an integer in $\big[\lfloor k-N\rfloor_+,k\big]$ satisfying
\begin{equation*}
F(x^{\iota(k)};\mu)=\max_{\lfloor k-N\rfloor_+\leq j\leq k}F(x^j;\mu),\forall k\geq0.
\end{equation*}
Based on (\ref{alg11}), we have
\begin{equation}\label{a1thm1}
F(x^{k+1};\mu_k)\leq F(x^{\iota(k)};\mu_k)-\frac{c}{2}\|x^{k+1}-x^k\|^2,\forall k\geq0.
\end{equation}
Then, $F(x^{\iota(k+1)};\mu_k)\leq\max\{F(x^{k+1};\mu_k),F(x^{\iota(k)};\mu_k)\}\leq F(x^{\iota(k)};\mu_k),\forall k\geq0$. Further, since $F(x)=F(x;\nu)$ and $\mu_k=\nu,\forall k\geq K$, we obtain
\begin{equation}\label{impieq}
F(x^{\iota(k+1)})\leq F(x^{\iota(k)}),\forall k\geq K,
\end{equation}and hence $F(x^{k+1})\leq F(x^{\iota(K)}),\forall k\geq K$ by (\ref{a1thm1}). Moreover, considering that $f$ is level-bounded on $\Omega$ by Assumption \ref{as-h} (i) and the form of $\theta$, we get that $\{x^k\}$ is bounded, which implies that $\{F(x^{\iota(k)})\}$ is bounded from below. Then, by (\ref{impieq}), we have that there exists an $\eta\in\mathbb{R}$ such that $\lim_{k\rightarrow\infty}F(x^{\iota(k)})=\eta$. It follows from (\ref{a1thm1}) that
\begin{equation}\label{a1thm2}
F(x^{k+1})\leq F(x^{\iota(k)})-\frac{c}{2}\|x^{k+1}-x^k\|^2,\forall k\geq K.
\end{equation}
Similar to the proof of Lemma 4 in \cite{Wright2009}, by (\ref{a1thm2}), we have $\lim_{k\rightarrow\infty}\|x^{k+1}-x^k\|=0$ and $\lim_{k\rightarrow\infty}F(x^k)=\eta$.
\end{proof}

In what follows, we show that all accumulation points of $\{x^k\}$ generated by Algorithm \ref{g0palg1} have a common support set and some entries of $x^k$ converge in finite iterations. Moreover, any accumulation point of $\{x^k\}$ is an sw-d-stationary point of (\ref{g0ps}) and also a $\nu$-strong local minimizer of (\ref{g0p}).
\begin{theorem}\label{gspthm1}
Let $\{x^k\}$ be the iterates generated by Algorithm \ref{g0palg1}. Then the following statements hold.\\
(i) $I^k$ and $J^k$ only change finite number of times and for any accumulation point $x^*$ and $\hat{x}$ of $\{x^k\}$, one has that $\mathcal{A}(x^*)=\mathcal{A}(\hat{x})$ and $x^k_{\mathcal{A}(x^*)}$ converges to $\bm{0}$ in finite number of iterations.\\
(ii) Any accumulation point of $\{x^k\}$ is an sw-d-stationary point of problem (\ref{g0ps}) and also a $\nu$-strong local minimizer of problem (\ref{g0p}).
\end{theorem}
\begin{proof}
(i) Let ${k}\geq K$. It follows from the first-order optimality condition of (\ref{alg12}) with ${\alpha}_k=\bar{\alpha}_k$, $x^k(\alpha_k)=x^{k+1}$ and $\xi^k\in\partial\Theta_{I^k,J^k}(x^k)$ that
\begin{small}
\begin{equation}\label{impincl}
\begin{split}
\bm{0}\in&\partial f_n(x^{k+1})+\nabla f_s(x^{k})+\frac{\lambda_1}{\nu}\partial\|x^{k+1}\|_1+\frac{\lambda_2}{\nu}\sum_{l=1}^L w_l\partial\|x^{k+1}_{(l)}\|_p\\&-\lambda_1\sum_{j=1}^{n}\nabla\theta_{I^k_j}(x^{k}_j)-\lambda_2\sum_{l=1}^{L}w_l\partial\theta_{J^k_l}(\|x^{k}_{(l)}\|_p)+\bar\alpha_{k}(x^{{k}+1}-x^{k})+N_{\Omega}(x^{k+1}).
\end{split}
\end{equation}
\end{small}
	
Recalling Proposition \ref{innerloops}, we have that $\bar{\alpha}_k\leq\tilde{\alpha}$ with $\tilde{\alpha}:=\max\{\bar{\alpha},\rho L_s\}$. Let $\bar{\nu}\in\big(\nu,\min\{\frac{\lambda_1}{L_f},\vartheta\}\big)$. Based on $\lim_{k\rightarrow\infty}\|x^{k+1}-x^k\|=0$ in Proposition \ref{alg1prop}, we have that there exists a ${\tilde{k}}$ such that for any $k\geq\tilde{k}$,
\begin{small}
\begin{equation}\label{impineq-}
\begin{aligned}
&\|x^{{k}+1}-x^{{k}}\|\\<&\left\{
\begin{split}
&\min\left\{\bar{\nu}-\nu,\frac{1}{2\tilde{\alpha}}\left(\frac{\lambda_1}{\nu}-L_f\right),\nu\frac{\lambda_1-\nu L_f}{\lambda_2{\sum_{l=1}^{L}}w_{l}}\Big/ \Big(\frac{\lambda_1-\nu L_f}{\lambda_2{\sum_{l=1}^{L}}w_{l}}+4\Big)\right\},&&\lambda_2>0, \sum_{l=1}^{L}w_{l}>0,\\
&\min\left\{\bar{\nu}-\nu,\frac{1}{2\tilde{\alpha}}\left(\frac{\lambda_1}{\nu}-L_f\right)\right\},&&\lambda_2=0~\mbox{or}\sum_{l=1}^{L}w_{l}=0.
\end{split}
\right.
\end{aligned}
\end{equation}
\end{small}It follows from Assumption \ref{gsassum} (ii), Assumption \ref{as-h} (iii) and (\ref{impineq-}) that for any $\xi\in\partial f_n(x^{{\hat{k}}+1})+\nabla f_s(x^{\hat{k}})$,
\begin{equation}\label{nec-exp}
\big|\xi_i+\bar{\alpha}_{\hat{k}}(x^{{\hat{k}}+1}_i-x^{\hat{k}}_i)\big|\leq|\xi_i|+\tilde{\alpha}\|x^{{\hat{k}}+1}-x^{\hat{k}}\|<L_f+\frac{1}{2}\big(\frac{\lambda_1}{\nu}-L_f\big)<\frac{\lambda_1}{\nu}.
\end{equation}
	
Suppose there exist $\hat{k}\geq\tilde{k}$ and ${i}\in[n]$ such that $|x^{\hat{k}}_{{i}}|\in[0,\nu)$. Then, $I^{{\hat{k}}}_{{i}}=1$, which means $\nabla\theta_{I_{{i}}^{\hat{k}}}(x_{{i}}^{\hat{k}})=0$. Next, we will prove $x^{{\hat{k}}+1}_{{i}}=0$ by contradiction. 
	
Suppose $x^{{\hat{k}}+1}_{{i}}\neq0$. By (\ref{impineq-}), we have $x^{{\hat{k}}+1}_{{i}}\in{\rm{int}} \Omega_{{i}}$, which means $[N_\Omega(x^{{\hat{k}}+1})]_{{i}}=\{0\}$. Further, let $\mathcal{L}_{{i}}:=\{l\in[L]:{i}\in G_{l}\}$ and $\mathcal{L}^{\nu}_{{i}}:=\{l\in[L]:{i}\in G_{l},\|x^{\hat{k}}_{({l})}\|_p\geq\nu\}$, we have $\nabla\theta_{J^{\hat{k}}_l}(\|x^{\hat{k}}_{(l)}\|_p)=\bm{0},\forall l\in\mathcal{L}_{{i}}\setminus\mathcal{L}^{\nu}_{{i}}$. We will prove a contradiction of (\ref{impincl}), that is
\begin{equation}\label{case2notincl2}
0\notin\big[\partial f_n(x^{{\hat{k}}+1})+\nabla f_s(x^{\hat{k}})\big]_i+\bar{\alpha}_{\hat{k}}(x^{{\hat{k}}+1}_i-x^{\hat{k}}_i)+\frac{\lambda_1}{\nu}\nabla_{x_i}|x^{{\hat{k}}+1}_i|+\sum_{l\in\mathcal{L}_i}\frac{\lambda_2 w_{l}}{\nu}\nabla_{x_i}\|x^{\hat{k}+1}_{(l)}\|_p-\sum_{l\in\mathcal{L}^{\nu}_i}\frac{\lambda_2 w_{l}}{\nu}\partial_{x_i}\|x^{\hat{k}}_{(l)}\|_p.
\end{equation}
Since $x^{{\hat{k}}+1}_{{i}}\neq0$ and $\nabla_{x_{{i}}}\|x^{\hat{k}+1}_{(l)}\|_p,l\in\mathcal{L}_i$ have the same sign as $\nabla_{x_{{i}}}|x^{{\hat{k}}+1}_{{i}}|$, we obtain that for any $P\subseteq\mathcal{L}_i$,
\begin{equation}\label{unif-exp}
\Big|\frac{\lambda_1}{\nu}\nabla_{x_i}|x^{{\hat{k}}+1}_i|+\sum_{l\in P}\frac{\lambda_2 w_{l}}{\nu}\nabla_{x_i}\|x^{\hat{k}+1}_{(l)}\|_p\Big|\geq\frac{\lambda_1}{\nu}.
\end{equation}When $\lambda_2=0$ or $\sum_{l=1}^{L}w_{l}=0$, we have (\ref{case2notincl2}) holds by (\ref{nec-exp}) and $\big|\frac{\lambda_1}{\nu}\nabla_{x_i}|x^{{\hat{k}}+1}_i|\big|=\frac{\lambda_1}{\nu}$. Next, we consider the case $\lambda_2>0$ and $\sum_{l}^{L}w_{l}>0$ by the situations of $\mathcal{L}^{\nu}_{{i}}=\emptyset$ and $\mathcal{L}^{\nu}_{{i}}\neq\emptyset$. If $\mathcal{L}^{\nu}_{{i}}=\emptyset$, by (\ref{nec-exp}) and (\ref{unif-exp}) with $P=\mathcal{L}_i$, we obtain that for any $\xi\in\partial f_n(x^{{\hat{k}}+1})+\nabla f_s(x^{\hat{k}})$,
\begin{equation*}
\Big|\xi_i+\bar{\alpha}_{\hat{k}}(x^{{\hat{k}}+1}_i-x^{\hat{k}}_i)\Big|<\Big|\frac{\lambda_1}{\nu}\nabla_{x_i}|x^{{\hat{k}}+1}_i|+\sum_{l\in\mathcal{L}_i}\frac{\lambda_2 w_{l}}{\nu}\nabla_{x_i}\|x^{\hat{k}+1}_{(l)}\|_p\Big|,
\end{equation*}which implies that (\ref{case2notincl2}) holds. If $\mathcal{L}^{\nu}_{{i}}\neq\emptyset$, we will prove (\ref{case2notincl2}) holds for $p=1$ and $p=2$, respectively. Firstly, consider $p=1$. Based on the fact that for any $\zeta\in\nabla_{x_{{i}}}\|x^{\hat{k}+1}_{(l)}\|_1-\partial_{x_{{i}}}\|x^{\hat{k}}_{(l)}\|_1$, $\zeta$ is $0$ or has the same sign as $\nabla_{x_{{i}}}|x^{{\hat{k}}+1}_{{i}}|$ and $\nabla_{x_{{i}}}\|x^{\hat{k}+1}_{(l)}\|_1$, we have $\big|\frac{\lambda_1}{\nu}\nabla_{x_i}|x^{{\hat{k}}+1}_i|+\sum_{l\in\mathcal{L}_i}\frac{\lambda_2 w_{l}}{\nu}\nabla_{x_i}\|x^{\hat{k}+1}_{(l)}\|_1-\sum_{l\in\mathcal{L}^{\nu}_i}\frac{\lambda_2 w_{l}}{\nu}\partial_{x_i}\|x^{\hat{k}}_{(l)}\|_p\big|\geq\frac{\lambda_1}{\nu}$ by (\ref{unif-exp}) with $P=\mathcal{L}_i\setminus \mathcal{L}^{\nu}_i$, which implies that (\ref{case2notincl2}) holds by (\ref{nec-exp}). Next, considering $p=2$, we will prove
\begin{equation}\label{nm2leq}
\sum_{l\in\mathcal{L}^{\nu}_{{i}}}\frac{\lambda_2 w_{l}}{\nu}\big|\nabla_{x_{{i}}}\|x^{\hat{k}+1}_{(l)}\|-\nabla_{x_{{i}}}\|x^{\hat{k}}_{(l)}\|\big|<\frac{1}{2}\big(\frac{\lambda_1}{\nu}-L_f\big),
\end{equation}which implies (\ref{case2notincl2}) holds by (\ref{nec-exp}) and (\ref{unif-exp}) with $P=\mathcal{L}_i\setminus \mathcal{L}^{\nu}_i$. By (\ref{impineq-}) and $\hat{k}\geq\tilde{k}$, one has $\|x^{\hat{k}+1}-x^{\hat{k}}\|<\nu\frac{\lambda_1-\nu L_f}{\lambda_2{\sum_{l=1}^{L}}w_{l}}\Big/ \Big(\frac{\lambda_1-\nu L_f}{\lambda_2{\sum_{l=1}^{L}}w_{l}}+4\Big):=\varepsilon<\nu$. Take $\hat{l}\in\mathcal{L}^{\nu}_{{i}}$. Then, $\|x^{\hat{k}+1}_{(\hat{l})}\|>\|x^{\hat{k}}_{(\hat{l})}\|-\varepsilon\geq\nu-\varepsilon$. Since $\nabla_{x_{{i}}}\|x_{(\hat{l})}\|$ is Lipschitz continuous with modulus $\frac{2}{\nu-\varepsilon}$ in $\{x:\|x_{(\hat{l})}\|\geq\nu-\varepsilon\}$, we deduce that
\begin{equation}\label{ineq}
\begin{split}
\big|\nabla_{x_{{i}}}\|x^{{\hat{k}}+1}_{(\hat{l})}\|-\nabla_{x_{{i}}}\|x^{\hat{k}}_{(\hat{l})}\|\big|\leq\frac{2}{\nu-\varepsilon}|x^{\hat{k}}_i-x^{\hat{k}+1}_i|<\frac{2\varepsilon}{\nu-\varepsilon}.
\end{split}
\end{equation}
Then $\sum_{l\in\mathcal{L}^{\nu}_{{i}}}\frac{\lambda_2 w_{l}}{\nu}\big|\nabla_{x_{{i}}}\|x^{{\hat{k}}+1}_{(l)}\|-\nabla_{x_{{i}}}\|x^{\hat{k}}_{(l)}\|\big|<\frac{\lambda_2 \sum_{l=1}^Lw_{l}}{\nu}\frac{2\varepsilon}{\nu-\varepsilon}=\frac{1}{2}(\frac{\lambda_1}{\nu}-L_f)$, which means that (\ref{nm2leq}) holds. Therefore, if $x^{{\hat{k}}+1}_{{i}}\neq0$, we have (\ref{case2notincl2}) holds, but contradicts (\ref{impincl}). 
	
In conclusion, if there exist $\hat{k}\geq\tilde{k}$ and ${i}\in[n]$ such that $|x^{\hat{k}}_{{i}}|\in[0,\nu)$, then $x^{{\hat{k}}+1}_{{i}}=0$. By induction, we deduce that $x^k_{{i}}=0,\forall k>{\hat{k}}$. Therefore, for any $j\in[n]$,
\begin{equation}\label{lwprac}
\mbox{$\exists$ $\bar{k}\geq\tilde{k}$ s.t. for any $k\geq\bar{k}$, $|x^k_{j}|$ is always $0$ or no less than $\nu$,}
\end{equation}
which implies that $I^k$ and $J^k$ only change finite times. Hence, for any accumulation point $x^*$ and $\hat{x}$ of $\{x^k\}$, we have that $\mathcal{A}(x^*)=\mathcal{A}(\hat{x})$ and $x^k_{\mathcal{A}(x^*)}$ converges to $\bm{0}$ in finite iterations.
	
(ii) Since $\{x^k\}$ is bounded, we have that there exists at least an accumulation point of $\{x^k\}$, denoted by $x^*$. By (i), there exists a $\bar{k}\geq\tilde{k}$ such that for any $k\geq\bar{k}$, $I^k=I^{x^*}$ and $J^k=J^{x^*}$, which implies $\xi^k\in\partial\Theta_{I^{x^*},J^{x^*}}(x^k),\forall k\geq\bar{k}$. Then, by (\ref{alg12}), we have
\begin{equation*}
x^{k+1}=\underset{x\in\Omega}{\rm{argmin}}\{{w}_{I^{x^*},J^{x^*},\xi^{k}}(x;\nu,x^{k},x^{k})+\frac{\bar\alpha_{k_i}}{2}\|x-x^{k}\|^2\},\forall k\geq\bar{k}.
\end{equation*}
Since $x^*$ is an accumulation point of $\{x^k\}$, there exists a subsequence $\{x^{k_i}\}\subseteq\{x^k:k\geq\bar{k}\}$ satisfying $\lim_{k_i\rightarrow\infty}x^{k_i}=x^*$. Combining with $\lim_{k\rightarrow\infty}\|x^{k+1}-x^k\|=0$, we deduce $\lim_{k_i\rightarrow\infty}x^{k_i+1}=x^*$. Based on the boundedness of $\partial\Theta_{I^{x^*},J^{x^*}}$, $\lim_{k_i\rightarrow\infty}x^{k_i}=x^*$, $\bar\alpha_k\in[\underline{\alpha},\max\{\bar{\alpha},\rho L_s\}]$ in Proposition \ref{innerloops} and the upper semi-continuity of $\partial\Theta_{I^{x^*},J^{x^*}}$ in Proposition 2.1.5 (b) of \cite{Clarke1983}, there exist $\xi^*\in\mathbb{R}^n$, $\alpha_*\in[\underline{\alpha},\max\{\bar{\alpha},\rho L_s\}]$ and a subsequence $\{k_i\}$ (also denoted by $\{k_i\}$) such that $\lim_{k_i\rightarrow\infty}\xi^{k_i}=\xi^*\in\partial\Theta_{I^{x^*},J^{x^*}}(x^*)$ and $\lim_{k_i\rightarrow\infty}\bar{\alpha}_{k_i}=\alpha_*$. Therefore, by Proposition \ref{coro1} and $\lim_{k_i\rightarrow\infty}x^{k_i+1}=x^*$, we obtain \begin{equation*}x^*=\underset{x\in\Omega}{\rm{argmin}}\{{w}_{I^{x^*},J^{x^*},\xi^*}(x;\nu,x^*,x^*)+\frac{\alpha_*}{2}\|x-x^*\|^2\}.
\end{equation*}
Further, by its first-order optimality condition, we have
\begin{equation}\label{alg1alp}
\bm{0}\in\partial\bar{f}(x^*)-\lambda_1\sum_{j=1}^{n}\nabla\theta_{I^{x^*}_j}(x^{*}_j)-\lambda_2\sum_{l=1}^{L}w_l\partial\theta_{J^{x^*}_l}(\|x^*_{(l)}\|_p)+N_{\Omega}(x^*).
\end{equation}
	
When $p=2$, since $\partial\theta_{J^{x^*}_l}(\|x^*_{(l)}\|)=\{\nabla\theta_{J^{x^*}_l}(\|x^*_{(l)}\|)\},l\in[L]$, by rearranging the terms in (\ref{alg1alp}), we have that $x^*$ satisfies (\ref{swdstat}) with $p=2$. 
	
Next, consider the case of $p=1$. Let $\mathcal{L}_j:=\{l\in[L]:j\in G_l\},\forall j\in[n]$. Based on (\ref{lwprac}), we deduce that $x^*$ has $\nu$-lower bound property. Then, we have that for any $j\in\mathcal{A}^c(x^*)$, $\partial_{x_j}\theta_{J^{x^*}_l}(\|x^*_{(l)}\|_1)=\{\sgn(x^*_j)/\nu\}=\partial_{x_j}\|x^*_{(l)}\|_1/\nu,\forall l\in\mathcal{L}_j$, which implies $\big[\lambda_2\sum_{l=1}^{L}w_l\partial\theta_{J^{x^*}_l}(\|x^*_{(l)}\|_1)\big]_{\mathcal{A}^c(x^*)}=\big[\frac{\lambda_2}{\nu}\sum_{l=1}^{L}w_l\partial\|x^*_{(l)}\|_1\big]_{\mathcal{A}^c(x^*)}$. Then, by (\ref{alg1alp}), we have that there exists an $\eta^*\in\partial{f}(x^*)$ such that
\begin{equation}\label{incl-stat-pf1}
\Big[\lambda_1\sum_{i=1}^n\nabla\theta_{I^{x^*}_j}(x^*_j)\Big]_{\mathcal{A}^c(x^*)}\in\Big[\eta^*+\frac{\lambda_1}{\nu}\partial\|x^*\|_1+N_{\Omega}(x^*)\Big]_{\mathcal{A}^c(x^*)}.
\end{equation}
By Assumption \ref{as-h} and $\bm{0}\in N_{\Omega}(x^*)$, we have $\bm{0}\in\big[\eta^*+\frac{\lambda_1}{\nu}\partial\|x^*\|_1+N_{\Omega}(x^*)\big]_{\mathcal{A}(x^*)}$. Combining with $\big[\lambda_1\sum_{j=1}^{n}\nabla\theta_{I^{x^*}_j}(x^*_j)\big]_{\mathcal{A}(x^*)}=\bm{0}$, we obtain $\big[\lambda_1\sum_{i=1}^n\nabla\theta_{I^{x^*}_j}(x^*_j)\big]_{\mathcal{A}(x^*)}\in\big[\eta^*+\frac{\lambda_1}{\nu}\partial\|x^*\|_1+N_{\Omega}(x^*)\big]_{\mathcal{A}(x^*)}$. It then follows from (\ref{incl-stat-pf1}), $\partial\|x^*\|_1=[\partial\|x^*\|_1]_{\mathcal{A}(x^*)}\times[\partial\|x^*\|_1]_{\mathcal{A}^c(x^*)}$, $N_{\Omega}(x^*)=[N_{\Omega}(x^*)]_{\mathcal{A}(x^*)}\times[N_{\Omega}(x^*)]_{\mathcal{A}^c(x^*)}$ and $\eta^*\in\partial{f}(x^*)$ that
\begin{equation}\label{reltac}
\lambda_1\sum_{i=1}^n\nabla\theta_{I^{x^*}_j}(x^*_j)\in\partial f(x^*)+\frac{\lambda_1}{\nu}\partial\|x^*\|_1+N_{\Omega}(x^*).
\end{equation}
Considering the $\nu$-lower bound property of $x^*$, we have $\lambda_2\sum_{l=1}^{L}w_l\partial\theta_{J^x_l}(\|x^*_{(l)}\|_1)\subseteq\frac{\lambda_2}{\nu}\sum_{l=1}^{L}w_l\partial\|x^*_{(l)}\|_1$. Adding them to both sides of (\ref{reltac}), respectively, we obtain $\lambda_1\sum_{j=1}^{n}\nabla\theta_{I^{x^*}_j}(x^*_j)+\lambda_2\sum_{l=1}^{L}w_l\partial\theta_{J^{x^*}_l}(\|x^*_{(l)}\|_1)\subseteq\partial\bar{f}(x^*)+N_{\Omega}(x^*)$, which means that $x^*$ satisfies (\ref{swdstat}) with $p=1$.
	
In conclusion, $x^*$ is an sw-d-stationary point of problem (\ref{g0ps}) and also a $\nu$-strong local minimizer of problem (\ref{g0p}) by Proposition \ref{local-stat}.
\end{proof}

\subsection{A DC algorithm with extrapolation}
In this part, we design a DC algorithm with extrapolation for solving (\ref{g0ps}), which is presented in Algorithm \ref{g0palg2}.

\begin{algorithm}
\caption{DC algorithm with extrapolation} \label{g0palg2}
{\bf Initialization:}
Choose $x^0\in\Omega$ and $\beta\in[0,1)$. Set $\{\mu_k\}$ be defined as in (\ref{mukexp}), $x^{-1}=x^0$ and $k=0$.
	
{\bf Step 1:} Choose $\beta_k\in[0,\beta]$ arbitrarily. Set $y^k=x^k+\beta_k(x^k-x^{k-1})$.
	
{\bf Step 2:} Take $\xi^k\in\partial_x\Theta_{I^k,J^k}(x^k;\mu_k)$ and compute $x^{k+1}$ by
\begin{equation}\label{alg21}
{x}^{k+1}=\underset{x\in\Omega}{\rm{argmin}}\Big\{w_{I^k,J^k,\xi^k}(x;\mu_k,y^k,x^k)+\frac{L_s}{2}\|x-y^k\|^2\Big\}.
\end{equation}
	
{\bf Step 3:} Update $k\leftarrow k+1$ and return to Step 1.
\end{algorithm}

The subproblem (\ref{alg21}) in Algorithm \ref{g0palg2} is equivalent to the following form
\begin{equation*}
x^{k+1}={\rm{prox}}_{\frac{1}{L_s}\big(\bar{f}_n(\cdot;\mu_k)+\delta_\Omega(\cdot)\big)}\Big(y^k-\frac{1}{L_s}\big(\nabla f_s(y^k)-\xi^k\big)\Big).
\end{equation*}Its calculation is given in Remark \ref{alg1-sp} with $\zeta^k=y^k-\frac{1}{L_s}\big(\nabla f_s(y^k)-\xi^k\big)$.

Firstly, we give some basic properties of Algorithm \ref{g0palg2}.
\begin{proposition}\label{alg2prop}
Let $\{x^k\}$ be generated by Algorithm \ref{g0palg2}. Then we have that $\{x^k\}$ is bounded, $\lim_{k\rightarrow\infty}\|x^k-x^{k-1}\|=0$ and $\lim_{k\rightarrow\infty}F(x^k)$ exists.
\end{proposition}
\begin{proof}
Firstly, we have that for any $k\geq K$,
\begin{equation*}
\begin{split}
&F(x^k)=\bar{f}_n(x^k;\nu)+f_s(x^k)-\Theta_{I^k,J^k}(x^k)\\\geq&\bar{f}_n(x^k;\nu)+f_s(y^k)+\langle\nabla f_s(y^k),x^k-y^k\rangle-\Theta_{I^k,J^k}(x^k)-\langle\xi^k,x^k-x^k\rangle\\=&{w}_{I^k,J^k,\xi^k}(x^k;\nu,y^k,x^k)\\\geq& w_{I^k,J^k,\xi^k}(x^{k+1};\nu,y^k,x^k)+\frac{L_s}{2}\|x^{k+1}-y^k\|^2+\frac{L_s}{2}\|x^k-x^{k+1}\|^2-\frac{L_s}{2}\|x^k-y^k\|^2\\=&\bar{f}_n(x^{k+1};\nu)+f_s(y^k)+\langle\nabla f_s(y^k),x^{k+1}-y^k\rangle+\frac{L_s}{2}\|x^{k+1}-y^k\|^2-\Theta_{I^k,J^k}(x^k)\\&-\langle\xi^k,x^{k+1}-x^k\rangle+\frac{L_s}{2}\|x^k-x^{k+1}\|^2-\frac{L_s}{2}\|x^k-y^k\|^2\\\geq& \bar{f}_n(x^{k+1};\nu)+f_s(x^{k+1})-\Theta_{I^k,J^k}(x^{k+1})+\frac{L_s}{2}\|x^k-x^{k+1}\|^2-\frac{L_s}{2}\|x^k-y^k\|^2\\\geq& F(x^{k+1})+\frac{L_s}{2}\|x^k-x^{k+1}\|^2-\frac{L_s}{2}\|x^k-y^k\|^2,
\end{split}
\end{equation*}
where the first equality and last inequality are due to the definition of $\Theta_{I^k,J^k}$, the first inequality follows from the convexity of $f_s$, the second and third equalities follow from the definition of ${w}_{I,J,\xi}$, the second inequality holds because the objective function in (\ref{alg21}) is strongly convex with modulus $L_s$ and $x^{k+1}$ is a minimizer of (\ref{alg21}), and the third inequality follows from the Lipschitz continuity of $\nabla f_s$ and convexity of $\Theta_{I^k,J^k}$. Then by $y^k=x^k+\beta_k(x^k-x^{k-1})$ and $\beta\in[0,1)$, we obtain that
\begin{equation*}
F(x^{k+1})+\frac{L_s}{2}\|x^{k+1}-x^{k}\|^2\leq F(x^k)+\frac{L_s}{2}\|x^k-y^k\|^2\leq F(x^k)+\frac{L_s\beta^2}{2}\|x^{k}-x^{k-1}\|^2,
\end{equation*}
and hence
\begin{equation}\label{al2ieq2}
0\leq\frac{L_s(1-\beta^2)}{2}\|x^k-x^{k-1}\|^2\leq F(x^k)+\frac{L_s}{2}\|x^k-x^{k-1}\|^2-\big(F(x^{k+1})+\frac{L_s}{2}\|x^{k+1}-x^{k}\|^2\big).
\end{equation}
Therefore $F(x^{k})+\frac{L_s}{2}\|x^{k}-x^{k-1}\|^2$ is decreasing with respect to $k$ when $k\geq K$ and hence $F(x^{k})\leq F(x^{k})+\frac{L_s}{2}\|x^{k}-x^{k-1}\|^2\leq F(x^{K})+\frac{L_s}{2}\|x^{K}-x^{K-1}\|^2$ when $k\geq K$. Further, by Assumption \ref{as-h} (i) that $f$ is level-bounded, and the form of $\theta$, we have that $\{x^k\}$ is bounded, which implies that $F(x^k)+\frac{L_s}{2}\|x^{k}-x^{k-1}\|^2$ is bounded from below. Then, there exists an $\eta\in\mathbb{R}$ such that $\lim_{k\rightarrow\infty}F(x^k)+\frac{L_s}{2}\|x^k-x^{k-1}\|^2=\eta$. Further, considering (\ref{al2ieq2}), we deduce that $\lim_{k\rightarrow\infty}\|x^k-x^{k-1}\|=0$ and hence $\lim_{k\rightarrow\infty}F(x^k)=\eta$.
\end{proof}

Similar to the convergence analysis of Algorithm \ref{g0palg1}, we show that all accumulation points of $\{x^k\}$ generated by Algorithm \ref{g0palg2} have a common support set and their zero entries can be converged within finite iterations. Moreover, any accumulation point of $\{x^k\}$ is a $\nu$-strong local minimizer of (\ref{g0p}).
\begin{theorem}\label{g0pthm2}
Let $\{x^k\}$ be generated by Algorithm \ref{g0palg2}. Then, the statements (i) and (ii) in Theorem 3.5 hold.
\end{theorem}
\begin{proof}
(i) Based on $\lim_{k\rightarrow\infty}\|x^k-x^{k-1}\|=0$ in Proposition \ref{alg2prop} and $y^k=x^k+\beta_k(x^k-x^{k-1})$, we deduce that $\lim_{k\rightarrow\infty}\|y^k-x^{k}\|=0$ and hence $\lim_{k\rightarrow\infty}\|x^{k+1}-y^{k}\|=\lim_{k\rightarrow\infty}\|x^{k+1}-x^k+(x^k-y^{k})\|=0$.

Let $\bar{\nu}\in\big(\nu,\min\{\frac{\lambda_1}{L_f},\vartheta\}\big)$. By $\lim_{k\rightarrow\infty}\|x^{k}-x^{k-1}\|=0$ and $\lim_{k\rightarrow\infty}\|x^{k+1}-y^{k}\|=0$, we have that there exists a ${\tilde{k}}\geq K$ such that for any $k\geq\tilde{k}$,
\begin{small}
\begin{equation*}\label{al2impineq-}
\|x^{{k}+1}-x^{{k}}\|<\left\{\begin{split}
&\min\left\{\bar{\nu}-\nu,\nu\frac{\lambda_1-\nu L_f}{\lambda_2\sum_{l}^{L}w_{l}}\Big/ \Big(\frac{\lambda_1-\nu L_f}{\lambda_2\sum_{l}^{L}w_{l}}+4\Big)\right\},&&\lambda_2>0,\sum_{l}^{L}w_{l}>0,\\
&\bar{\nu}-\nu,&&\lambda_2=0~\mbox{or}\sum_{l}^{L}w_{l}=0,
\end{split}
\right.
\end{equation*}
\end{small}
and $\|x^{k+1}-y^{k}\|<\frac{1}{2L_s}\left(\frac{\lambda_1}{\nu}-L_f\right),\forall k\geq\tilde{k}$.
	
Suppose there exist $\hat{k}\geq\tilde{k}$ and $i\in [n]$ such that $|x^{\hat{k}}_{i}|\in[0,\nu)$. Similar to the proof idea in Theorem \ref{gspthm1} (i), we have $x^k_{i}=0,\forall k>{\hat{k}}$. Therefore, there exists a $\bar{k}>0$ such that for any $k\geq\bar{k}$, $|x^k_{j}|$ is always $0$ or not less than $\nu$, $j\in[n]$, which implies that $I^k$ and $J^k$ only change finite times. Then, for any accumulation point $x^*$ and $\hat{x}$ of $\{x^k\}$, we have that $\mathcal{A}(x^*)=\mathcal{A}(\hat{x})$ and $x^k_{\mathcal{A}(x^*)}$ converges to $\bm{0}$ in finite iterations.
	
(ii) In view of the boundedness of $\{x^k\}$, there exists at least an accumulation point in $\{x^k\}$. Let $x^*$ be an accumulation point of $\{x^k\}$. Then there exists a subsequence $\{x^{k_i}\}\subseteq\{x^k\}$ satisfying $\lim_{k_i\rightarrow\infty}x^{k_i}=x^*$. Owning to $\lim_{k\rightarrow\infty}\|x^{k}-x^{k-1}\|=0$ and $y^k=x^k+\beta_k(x^k-x^{k-1})$ again, we have $\lim_{k_i\rightarrow\infty}x^{k_i+1}=\lim_{k_i\rightarrow\infty}y^{k_i}=x^*$. 
	
Similar to the proof idea in Theorem \ref{gspthm1} (ii), we have that $x^*$ is an sw-d-stationary point of problem (\ref{g0ps}) and also a $\nu$-strong local minimizer of problem (\ref{g0p}).
\end{proof}
\begin{remark}
By Theorem \ref{gspthm1} (i) and Theorem \ref{g0pthm2} (i), we have that all accumulation points of the iterates generated by the proposed algorithms satisfy that the absolute values of their nonzero entries have a unified lower bound. This property has many advantages in sparse optimization. For example, it can distinguish zero and nonzero entries of coefficients effectively in sparse high-dimensional regression \cite{Chartrand2008, Huang2008}, and  can also produce closed contours and neat edges for restored images \cite{Chen2012}.
\end{remark}
\section{Global convergence analysis of the algorithms}\label{section4}
In this section, we firstly transform the subproblems of Algorithm \ref{g0palg1} and Algorithm \ref{g0palg2} into an equivalent low-dimensional strong convex problem after  finite iterations. Then, under the assumption of KL property, we prove the global convergence and convergence rate of Algorithm \ref{g0palg1} with $N=0$ and Algorithm \ref{g0palg2} for solving problem (\ref{g0p}). It is worth noting that the proposed KL assumption is naturally satisfied for some common loss functions in sparse regression.

Firstly, we show that the subproblems in Algorithm \ref{g0palg1} and Algorithm \ref{g0palg2} are equivalent to a strongly convex problem in a low dimensional space after finite iterations.

\begin{proposition}\label{tpequt}
Denote $x^*$ as an accumulation point of $\{x^k\}$ in Algorithm \ref{g0palg1} or Algorithm \ref{g0palg2} and $\check{\Omega}:=\{x\in\Omega:x_j=0,\forall~j\in\mathcal{A}(x^*)\}$. There exists a $\check{K}>0$ such that for any $k\geq\check{K}$, $\mathcal{A}(x^k)=\mathcal{A}(x^*)$, $I^{k}=I^{x^*}$, $J^{k}=J^{x^*}$, the subproblem (\ref{alg12}) in Algorithm \ref{g0palg1} with $\alpha_k=\bar{\alpha}_k$ is equivalent to
\begin{equation}\label{alg01}
x^{k+1}=\underset{x\in\check\Omega}{\rm{argmin}}\{{w}_{I^{x^*},J^{x^*},\xi^k}(x;\nu,x^k,x^k)+\frac{\bar\alpha_k}{2}\|x-x^k\|^2\},
\end{equation}
and the subproblem (\ref{alg21}) in Algorithm \ref{g0palg2} is equivalent to 
\begin{equation}\label{alg02}
{x}^{k+1}=\underset{x\in\check\Omega}{\rm{argmin}}\{{w}_{I^{x^*},J^{x^*},\xi^k}(x;\nu,y^k,x^k)+\frac{L_s}{2}\|x-y^k\|^2\}.
\end{equation}
\end{proposition}
\begin{proof}
Since the subproblems (\ref{alg12}) and (\ref{alg21}), and problems (\ref{alg01}) and (\ref{alg02}) are all strongly convex, they have a unique minimizer, respectively. By Theorem \ref{gspthm1} (i), we have that there exists a $\check{K}>K$ such that for any $k\geq\check{K}$, $\mathcal{A}(x^k)=\mathcal{A}(x^*)$, $I^{k}=I^{x^*}$ and $J^{k}=J^{x^*}$. Hence, the subproblem (\ref{alg12}) with $\alpha_k=\bar{\alpha}_k$ becomes $x^{k+1}=\underset{x\in\Omega}{\rm{argmin}}\{{w}_{I^{x^*},J^{x^*},\xi^k}(x;\nu,x^k,x^k)+\frac{\bar\alpha_k}{2}\|x-x^k\|^2\}$. Further, by $x^{k+1}\in\check{\Omega}$ and the uniqueness of minimizer to (\ref{alg12}) and (\ref{alg01}), we deduce that (\ref{alg12}) with $\alpha_k=\bar{\alpha}_k$ is equivalent to  (\ref{alg01}) for any $k\geq\check{K}$. Similarly, we have that (\ref{alg21}) is equivalent to  (\ref{alg02}) for any $k\geq\check{K}$.
\end{proof}

Recall the definition of Kurdyka-\L ojasiewicz (KL) property.
\begin{definition}\label{def1}\cite{Attouch2013}
We say that a proper closed function $h:\mathbb{R}^n\rightarrow(-\infty,\infty]$ has KL property at an $\hat{x}\in{\rm{dom}}\partial h$, if there are $a\in(0,\infty]$, a continuous concave function $\varphi:[0,a)\rightarrow[0,\infty)$ and a neighborhood $V$ of $\hat{x}$ such that
\begin{itemize}
\item[(i)] $\varphi$ is continuously differentiable on $(0,a)$, $\varphi(0)=0$, $\varphi'>0$ on $(0,a)$;
\item[(ii)] for any $x\in V$ with $h(\hat{x})<h(x)<h(\hat{x})+a$, the following inequality holds \begin{equation}\label{varphi}
\varphi'(h(x)-h(\hat{x})){\rm{dist}}(\bm{0},\partial h(x))\geq 1.
\end{equation}
\end{itemize}
\end{definition}
If $h$ has KL property at $\hat{x}\in{\rm{dom}}\partial h$ with $\varphi$ in (\ref{varphi}) chosen as $\varphi(t)=a_0t^{1-\alpha}$ for some $a_0>0$ and $\alpha\in[0,1)$, we say that $h$ has the KL property at $\hat{x}$ with exponent $\alpha$. If $h$ has KL property at each point of $\mbox{dom}\partial{h}$ (with exponent $\alpha$), we say that $h$ is a KL function (with exponent $\alpha$). 

Next, we give the global convergence analysis of Algorithm \ref{g0palg1} with $N=0$ and Algorithm \ref{g0palg2}. Define $H(x,y):=f(x)+\frac{L_s}{2}\|x-y\|^2+\delta_{\check{\Omega}}(x)$.
\begin{theorem}\label{convergence-Alg1}
Let $\{x^k\}$ be the iterates generated by Algorithm \ref{g0palg1} with $N=0$ (or Algorithm \ref{g0palg2}). If $f+\delta_{\check\Omega}$ (or $H$) is a KL function, then $x^k$ converges to a $\nu$-strong local minimizer $x^*$ of problem (\ref{g0p}) and $\{x^k\}$ has finite length, i.e. $\sum_{k=0}^{\infty}\|x^{k+1}-x^k\|<\infty$. Further, if $f+\delta_{\check{\Omega}}$ has KL property at $x^*$ (or $H$ has KL property at $(x^*,x^*)$) with exponent $\alpha\in[0,1)$, then we have the following convergence rate results.
\begin{itemize}
\item[(i)] If $\alpha=0$, then $\{x^k\}$ converges in finite number of iterations;
\item[(ii)] If $\alpha\in(0,\frac{1}{2}]$, then $\{x^k\}$ is R-linearly convergent to $x^*$, i.e. there exist $c>0$ and $q\in[0,1)$ such that $\|x^k-x^*\|\leq cq^k,\forall k\geq0$;
\item[(iii)] If $\alpha\in(\frac{1}{2},1)$, then $\{x^k\}$ is R-sublinearly convergent to $x^*$, i.e. there exists $c>0$ such that $\|x^k-x^*\|\leq ck^{-\frac{1-\alpha}{2\alpha-1}},\forall k\geq0$.
\end{itemize}
\end{theorem}
\begin{proof}
Let $x^*$ be an accumulation point of $\{x^k\}$ generated by Algorithm \ref{g0palg1} with $N=0$ or Algorithm \ref{g0palg2}. In view of $\lim_{k\rightarrow\infty}\|x^k-x^{k-1}\|=0$ in Proposition \ref{alg2prop}, we deduce that $(x^*,x^*)$ is an accumulation point of $\{(x^k,x^{k-1})\}$. Denote $\mathcal{A}:=\mathcal{A}(x^*)$, $\mathcal{A}^c:=\mathcal{A}^c(x^*)$. By Proposition \ref{tpequt}, we have that for any $k>\check{K}$, $I^k=I^{x^*}$, $J^k=J^{x^*}$, $\mathcal{A}(x^k)=\mathcal{A}$ and $x^k\in\check{\Omega}$. Further, considering the $\nu$-lower bound property of $x^*$, we deduce that for any $k>\check{K}$, $|x^k_j|\geq\nu,\forall j\in\mathcal{A}^c$. Then, there exists a constant $\sigma>0$ such that $F(x^k)=f(x^k)+\delta_{\check{\Omega}}(x^k)+\sigma,\forall k>\check{K}$.
Further, for Algorithm \ref{g0palg1}, by (\ref{a1thm2}) and $N=0$, we have
\begin{equation}\label{im1}
f(x^{k+1})+\delta_{\check{\Omega}}(x^{k+1})\leq f(x^{k})+\delta_{\check{\Omega}}(x^k)-\frac{c}{2}\|x^{k+1}-x^k\|^2,k>\check{K},
\end{equation}
and for Algorithm \ref{g0palg2}, by (\ref{al2ieq2}), we have
\begin{equation}\label{im12}
H(x^{k+1},x^k)\leq H(x^k,x^{k-1})-\frac{L_s(1-\beta^2)}{2}\|x^k-x^{k-1}\|^2,k>\check{K}.
\end{equation}
	
Moreover, since $x^*$ is an accumulation point of $\{x^k\}$, for Algorithm \ref{g0palg1}, by Proposition \ref{alg1prop}, we have
\begin{equation}\label{im2}
\lim_{k\rightarrow\infty}\big[f(x^k)+\delta_{\check{\Omega}}(x^k)\big]=f(x^*)+\delta_{\check{\Omega}}(x^*),
\end{equation}
and for Algorithm \ref{g0palg2}, since $(x^*,x^*)$ is an accumulation point of $\{(x^k,x^{k-1})\}$, by Proposition \ref{alg2prop}, we deduce
\begin{equation}\label{im22}
\lim_{k\rightarrow\infty}H(x^k,x^{k-1})=H(x^*,x^*).
\end{equation}
	
For any $x\in\check{\Omega}$, we deduce that\begin{equation*}
\begin{split}
{w}_{I^{x^*},J^{x^*},\xi^k}(x;\nu,x^k,x^k)+\frac{\bar\alpha_k}{2}\|x-x^k\|^2=&f_n(x)+\frac{\lambda_1}{\nu}\|x_{\mathcal{A}^c}\|_1+\frac{\lambda_2}{\nu}\sum_{l=1}^{L}w_l\|{x_{\mathcal{A}^c}}_{(l)}\|_p+f_s(x^k)+\langle\nabla f_s(x^k),x-x^k\rangle\\&-\Theta_{I^{x^*},J^{x^*}}(x^k;\mu)-\langle\nabla_{x_{\mathcal{A}^c}} \Theta_{I^{x^*},J^{x^*}}(x^k),x_{\mathcal{A}^c}-x^k_{\mathcal{A}^c}\rangle+\frac{\bar\alpha_k}{2}\|x_{\mathcal{A}^c}-x^k_{\mathcal{A}^c}\|^2,
\end{split}
\end{equation*}
where ${x_{\mathcal{A}^c}}_{(l)}$ is the restriction of $x_{\mathcal{A}^c}$ to the index set $G_l\cap\mathcal{A}^c$. Then, by the first-order optimal condition of (\ref{alg01}) and the definition of $\Theta_{I,J}$, for any $k>\check{K}$, we have $\bm{0}\in [\partial ({f}_n+\delta_{\check\Omega})(x^{{k}+1})+\nabla f_s(x^{k})]_{\mathcal{A}}$ and
\begin{equation*}
\bm{0}\in[\partial ({f}_n+\delta_{\check\Omega})(x^{{k}+1})+\nabla f_s(x^{k})]_{\mathcal{A}^c}+\nabla_{x_{\mathcal{A}^c}} \Theta_{I^{x^*},J^{x^*}}(x^{k+1})-\nabla_{x_{\mathcal{A}^c}} \Theta_{I^{x^*},J^{x^*}}(x^k)+\bar\alpha_{k}(x^{{k}+1}_{\mathcal{A}^c}-x^{k}_{\mathcal{A}^c}).
\end{equation*}
Then, for Algorithm \ref{g0palg1}, we deduce that for any $k>\check{K}$,
\begin{equation*}
[\partial (f+\delta_{\check\Omega}) (x^{k+1})]_{\mathcal{A}}=[\partial ({f}_n+\delta_{\check\Omega})(x^{{k}+1})]_{{\mathcal{A}}}+[\nabla f_s(x^{k+1})]_{{\mathcal{A}}}\ni-[\nabla f_s(x^{k})]_{{\mathcal{A}}}+[\nabla f_s(x^{k+1})]_{{\mathcal{A}}}:=\zeta_{\mathcal{A}}^{k+1}.
\end{equation*}
and
\begin{equation*}
\begin{split}
&[\partial (f+\delta_{\check\Omega}) (x^{k+1})]_{\mathcal{A}^c}=[\partial ({f}_n+\delta_{\check\Omega})(x^{{k}+1})]_{{\mathcal{A}^c}}+[\nabla f_s(x^{k+1})]_{{\mathcal{A}^c}}\\\ni& -[\nabla f_s(x^{k})]_{{\mathcal{A}^c}}+\nabla_{x_{{\mathcal{A}^c}}} \Theta_{I^{x^*},J^{x^*}}(x^k)-\nabla_{x_{{\mathcal{A}^c}}} \Theta_{I^{x^*},J^{x^*}}(x^{k+1})-\bar\alpha_{k}(x^{{k}+1}_{{\mathcal{A}^c}}-x^{k}_{{\mathcal{A}^c}})+[\nabla f_s(x^{k+1})]_{{\mathcal{A}^c}}:=\zeta_{\mathcal{A}^c}^{k+1}.
\end{split}
\end{equation*}
Then, ${\zeta}^{k+1}\in\partial (f+\delta_{\check\Omega}) (x^{k+1}),\forall k>\check{K}$.
	
Similarly, for Algorithm \ref{g0palg2}, by the first-order optimal condition of (\ref{alg02}), we deduce that for any $k>\check{K}$,
$[\partial (f+\delta_{\check\Omega}) (x^{k+1})]_{\mathcal{A}}\ni-[\nabla f_s(y^{k})]_{{\mathcal{A}}}+[\nabla f_s(x^{k+1})]_{{\mathcal{A}}}:=\zeta_{\mathcal{A}}^{k+1}$
and $[\partial (f+\delta_{\check\Omega}) (x^{k+1})]_{\mathcal{A}^c}\ni-[\nabla f_s(y^{k})]_{{\mathcal{A}^c}}+\nabla_{x_{{\mathcal{A}^c}}} \Theta_{I^{x^*},J^{x^*}}(x^k)-\nabla_{x_{{\mathcal{A}^c}}} \Theta_{I^{x^*},J^{x^*}}(x^{k+1})-L_s(x^{{k}+1}_{{\mathcal{A}^c}}-y^{k}_{{\mathcal{A}^c}})+[\nabla f_s(x^{k+1})]_{{\mathcal{A}^c}}:=\zeta_{\mathcal{A}^c}^{k+1}$. Let
\begin{equation*}
\hat{\zeta}^{k+1}:=({\zeta}^{k+1}+L_s(x^{k+1}-x^k),L_s(x^k-x^{k+1})),\forall k>\check{K},
\end{equation*}then $\hat{\zeta}^{k+1}\in\partial H(x^k,x^{k+1}),\forall k>\check{K}$.
	
When $p=1$, by $\lim_{k\rightarrow\infty}\|x^{k}-x^{k-1}\|=0$, there exists a $\tilde{k}>\check{K}$ such that for any $k\geq\tilde{k}$, $\sgn(x^{k+1}_j)=\sgn(x^k_j),j\in\mathcal{A}^c$, which implies that $\|\nabla_{x_{{\mathcal{A}^c}}} \Theta_{I^{x^*},J^{x^*}}(x^k)-\nabla_{x_{{\mathcal{A}^c}}} \Theta_{I^{x^*},J^{x^*}}(x^{k+1})\|=0,\forall k\geq\tilde{k}$. When $p=2$, since $|x^k_j|\geq\nu,j\in\mathcal{A}^c$ for any $k>\check{K}$ and $\nabla_{x_i}\|x_{(l)}\|$ is global Lipschitz continuous in $\{x:\|x_{(l)}\|\geq\nu\}$, we deduce that there exists a $\bar\gamma>0$ such that $\|\nabla_{x_{{\mathcal{A}^c}}} \Theta_{I^{x^*},J^{x^*}}(x^k)-\nabla_{x_{{\mathcal{A}^c}}} \Theta_{I^{x^*},J^{x^*}}(x^{k+1})\|\leq\bar\gamma\|x^{k+1}-x^k\|,\forall k>\check{K}$. Further, by the Lipschitz continuity of $\nabla f_s$, $y^k=x^k+\beta_k(x^k-x^{k-1})$ and the boundedness of $\bar{\alpha}_k$ and $\beta_k$, there exist $d_1,d_2>0$ such that \begin{equation}\label{im3}
\|\zeta^{k+1}\|\leq d_1\|x^{k+1}-x^k\|,\forall k\geq\tilde{k},
\end{equation}
\begin{equation}\label{im32}
\|\hat\zeta^{k+1}\|\leq d_2(\|x^{k+1}-x^k\|+\|x^k-x^{k-1}\|),\forall k\geq\tilde{k}+1.
\end{equation}
Moreover, by Theorem \ref{gspthm1}, $x^*$ is a $\nu$-strong local minimizer of problem (\ref{g0p}) and hence $x^*$ is a global minimizer of $f$ on $\check\Omega$. Then, by (\ref{im1}), (\ref{im2}) and (\ref{im3}), based on Theorem 2.9 in \cite{Attouch2013}, or by (\ref{im12}), (\ref{im22}) and (\ref{im32}), similar to the proof of Theorem 4.2 (iv) in \cite{Wen2018}, we deduce that $\{x^k\}$ converges to the $\nu$-strong local minimizer $x^*$ of problem (\ref{g0p}) and $\sum_{k=0}^{\infty}\|x^{k+1}-x^k\|<\infty$. Similar to the proof of Theorem 2 in \cite{Attouch2009} or Theorem 4.3 in \cite{Wen2018}, since $f+\delta_{\check{\Omega}}$ has KL property at $x^*$ or $H$ has KL property at $(x^*,x^*)$ with an exponent $\alpha\in[0,1)$, we have results (i) (ii) and (iii) for any $k\geq\tilde{k}+1$. Then, adjusting $c$ by the former $\tilde{k}$ terms, we deduce the convergence rate results (i) (ii) and (iii).
\end{proof}
\begin{remark}
For the general DC minimization $\min_{x\in\mathbb{R}^n}h(x)-g(x)$, the global convergence analysis of DC algorithms usually depends on the Lipschitz continuous differentiability of the subtracted function $g$ \cite{Tang2020,Wen2018}. If $g(x)=\max_i\psi_i(x)$, the global convergence is proved based on Lipschitz continuous differentiability of $\psi_i$ \cite{Lu2019MP}. In this paper, the subtracted function in (\ref{g0ps}) does not own the above properties. Moreover, the proposed KL assumption is only assumed on the loss function $f$, but not on the objective function $F$ in (\ref{g0ps}).
\end{remark}

Indeed, the KL assumptions in Theorem \ref{convergence-Alg1} are not restrictive. We will show that there exist some common loss functions to satisfy that $f+\delta_{\check{\Omega}}$ or $H$ has the KL property with exponent $\frac{1}{2}$.
\begin{proposition}\label{con-Alg1}
Let $\{x^k\}$ be the iterates generated by Algorithm \ref{g0palg1} with $N=0$ or Algorithm \ref{g0palg2}. If $f=l(Ax)$, where $l$ is strongly convex and twice continuously differentiable and $A\in\mathbb{R}^{m\times n}$, then $\{x^k\}$ is R-linearly convergent to a $\nu$-strong local minimizer of problem (\ref{g0p}) and $\{x^k\}$ has finite length.
\end{proposition}
\begin{proof}
Since $\delta_{\check{\Omega}}$ is a proper closed polyhedral function, by Corollary 5.1 in \cite{Li2018}, we obtain that $f+\delta_{\check{\Omega}}$ is a KL function with exponent $\frac{1}{2}$. Further, by Theorem 3.6 in \cite{Li2018}, $H$ has KL property at each point in $\{(x,x):x\in\check{\Omega}\}$ with exponent $\frac{1}{2}$. Then, by Theorem \ref{convergence-Alg1}, the results in this proposition holds.
\end{proof}
\begin{remark}
In Proposition \ref{con-Alg1}, the assumptions on $l$ are general enough to cover the loss functions for linear, logistic and Poisson regression, where $l(y)=\|y-b\|^2$ with $b\in\mathbb{R}^n$, $l(y)=\sum_{i=1}^m\log(1+\exp(-b_iy_i))$ with $b\in\{-1,1\}^m$ and $l(y)=\sum_{i=1}^m(-b_iy_i+\exp(y_i))$ with $b\in\mathbb{N}^m$.
\end{remark}

\begin{remark}
For the simplicity of parameters, we use the capped-$\ell_1$ function with the same parameter $\nu$ to relax both the element-wise sparsity term $\mathcal{I}(|x_i|)$ and group-wise sparsity term $\mathcal{I}(\|x_{(l)}\|)$ in (\ref{g0p}). Based on the theoretical analysis throughout this paper, for $\mathcal{I}(|x_i|)$, the parameter $\nu$ in its capped-$\ell_1$ relaxation needs to satisfy Assumption \ref{gsassum}, and for $\mathcal{I}(\|x_{(l)}\|)$, the parameter in its relaxation can be any positive constant no larger than that for $\mathcal{I}(|x_i|)$. In particular, for the sw-d-stationary point of (\ref{g0ps}), the value of $\nu$ in its $\nu$-lower bound property is decided by the parameter $\nu$ in capped-$\ell_1$ relaxation for $\mathcal{I}(|x_i|)$.
\end{remark}

\section{Numerical Experiments}\label{section6}
In this section, we show some numerical performance of Algorithm \ref{g0palg1} and Algorithm \ref{g0palg2} under the theoretical satisfaction of parameters in the algorithms. All our computational results are obtained by running MATLAB 2016b on a MacBook Pro (2.30 GHz, 8.00 GB RAM). In the following, Algorithm \ref{g0palg1} and Algorithm \ref{g0palg2} are denoted as Alg.\ref{g0palg1} and Alg.\ref{g0palg2}. For the original vector $x\in\mathbb{R}^n$ and recovered vector $y\in\mathbb{R}^n$, let MSE$=\|y-x\|^2/n$ denote the mean squared error of $y$ to $x$ and PSNR$=-10\lg$MSE denote the peak signal-to-noise ratio of $y$ to $x$.

The stopping criterion is set as $|F(x^k)-F(x^{k-1})|\leq10^{-15}$ and $10^{-7}$ in Subsections \ref{num-exp1} and \ref{num-exp2}, respectively. In Alg.\ref{g0palg1} and Alg.\ref{g0palg2}, we set $\bar{\alpha}=10^4,\underline{\alpha}=10^{-4},N=1,\alpha_k^B=c=L_s/4,\rho=2,\beta=0$, $\nu=0.99\cdot\min\{\lambda_1/L_f,\vartheta\}$ and $\mu_k=\max\{\bar{\mu}_k,\nu\}$ with the subsequent given sequence $\{\bar{\mu}_k\}$. When $f=\|Ax-b\|^2$, we let $L_s=2\|A^\top A\|$ and use the following MATLAB codes to generate $L_f$.\\
{\tt{t=abs(A)'*(abs(A)*ones(n,1));\\
Lf=2*max([norm($\vartheta$*t-A'*b,inf),norm(-$\vartheta$*t-A'*b,inf)]);}}

\subsection{Signal Recovery}\label{num-exp1}
In this experiment, we test the effectiveness of Alg.\ref{g0palg1} and Alg.\ref{g0palg2} in restoring the signals with noise. We use the same settings as in Example 2 of \cite{Bian2020} to randomly generate the original signal $x^\star$ with $\|x^\star\|_0=s$, sensing matrix $A\in\mathbb{R}^{m\times n}$, and observation $b\in\mathbb{R}^m$ for positive integers $n$, $m$ and $s$. The MATLAB codes to generate the data are as follows.\\
{\tt{index=randperm(n); index=index(1:s); $x^\star$=zeros(n,1); B=randn(n,m);\\
$x^\star$(index)=unifrnd(2,10,[s,1]); A=orth(B)'; b=A*$x^\star$+0.01*randn(m,1);}}

For restoring the signals with noise, we use Alg.\ref{g0palg1} with $\bar\mu_k=4-k/4$ and Alg.\ref{g0palg2} with $\bar\mu_k=4-k/10$ to solve the following model with $\lambda_1=1$,
\begin{equation}\label{g0pep1}
\min_{x\in[0,10]^n}~\|Ax-b\|^2+\lambda_1\|x\|_0.
\end{equation}

Firstly, we use the dimensions $n=160$, $m=80$, $s=16$ and initial point $x^0=1.97\cdot\textbf{1}$ (same as in Example 2 of \cite{Bian2020}) to generate $A$, $b$, $x^\star$ and $x^0$. Further, we test the effectiveness of Alg.\ref{g0palg1} and Alg.\ref{g0palg2} in high dimensions. We keep the above ratios of $n,m,s$ with $n=1600, 16000$ and use the above initial point. 

In all numerical experiments of this subsection, we randomly generate ten sets of data for numerical comparison and show the average values in the following tables. For Alg.\ref{g0palg1} and Alg.\ref{g0palg2}, denote $k_1$ and $k_2$ as the average numbers of outer loop iterations before termination, respectively. Let $m_k$ be the average number of inner loop iterations of Alg.\ref{g0palg1} in the $k$th test and define $\underline{k}_1:=(\sum_{k=1}^{10}m_k)/10$. Denote Time(s) as the average CPU runtime, MSE as the average MSE of the output solutions, dist as the average Euclidean distance between the output solutions and their previous iterates, nnz as the average number of the non-zero entries in the output solutions, respectively.

From Table \ref{g0ptb1}, it can be seen that Alg.\ref{g0palg1} and Alg.\ref{g0palg2} have similar experimental results regardless of the values of $n$ and the output iterates of the two algorithms have the same number of non-zero entries. For comparing the proposed two algorithms, since the average iterations of the inner loop of Alg.\ref{g0palg1} is close to 2, we regard the output solution by the $k_1$th outer loop as the $k$th output with $k=2k_1$. Fig.\ref{g0pdF} presents the convergence of of $|F(x^{k})-F(\bar{x}^1)|$ by Alg.\ref{g0palg1} and $|F(x^{k})-F(\bar{x}^2)|$ by Alg.\ref{g0palg2} in the cases of $n=160$, $n=1600$ and $n=16000$, where $\bar{x}^1$ and $\bar{x}^2$ are the output solutions of Alg.\ref{g0palg1} and Alg.\ref{g0palg2}, respectively, where the curves of $k^{-1}$, $k^{-2}$ and $k^{-3}$ are used to evaluate the convergence rates on the objective function values. From Fig.\ref{g0psls}, we see that the output signals of Alg.\ref{g0palg1} and Alg.\ref{g0palg2} almost coincide with the original signals and satisfy the $\nu$-lower bound property. Moreover, based on Definition \ref{nustronglocal} and Proposition \ref{g0plocalcondition}, by Fig.\ref{g0pdFx}, we can see that the iterates of Alg.\ref{g0palg1} and Alg.\ref{g0palg2} tend to the $\nu$-strong local minimizer of (\ref{g0p}).
\begin{table}
	\renewcommand{\arraystretch}{1}
	\caption{Numerical results under different dimensions of problem (\ref{g0pep1})}
	\label{g0ptb1}
	\centering \begin{tabular}{|c||c|c|c|c|c|}
		\hline
		Alg.\ref{g0palg1}&$k_1(\underline{k}_1)$& Time(s) & MSE & dist & nnz\\
		\hline
		$n=160$ &$42.0(1.9)$& $0.01$& $1.87\times10^{-5}$ & $6.14\times10^{-8}$& $16$\\
		\hline
		$n=1600$ &$42.5(1.9)$& $0.09$& $2.18\times10^{-5}$ & $2.00\times10^{-7}$& $160$\\
		\hline
		$n=16000$ &$42.8(1.9)$& $9.51$& $2.20\times10^{-5}$ & $5.27\times10^{-7}$& $1600$\\
		\hline\hline
		Alg.\ref{g0palg2}&$k_2$& Time(s) & MSE & dist & nnz\\
		\hline
		$n=160$&$83.0$& $0.01$& $1.87\times10^{-5}$ & $3.21\times10^{-8}$& $16$\\
		\hline
		$n=1600$&$89.3$& $0.13$& $2.18\times10^{-5}$ & $8.81\times10^{-8}$& $160$\\
		\hline
		$n=16000$&$90.8$& $13.77$& $2.20\times10^{-5}$ & $2.27\times10^{-7}$& $1600$\\
		\hline
	\end{tabular}
\end{table}
\begin{figure}
	\centering
	\subfigure[]{\includegraphics[width=2in]{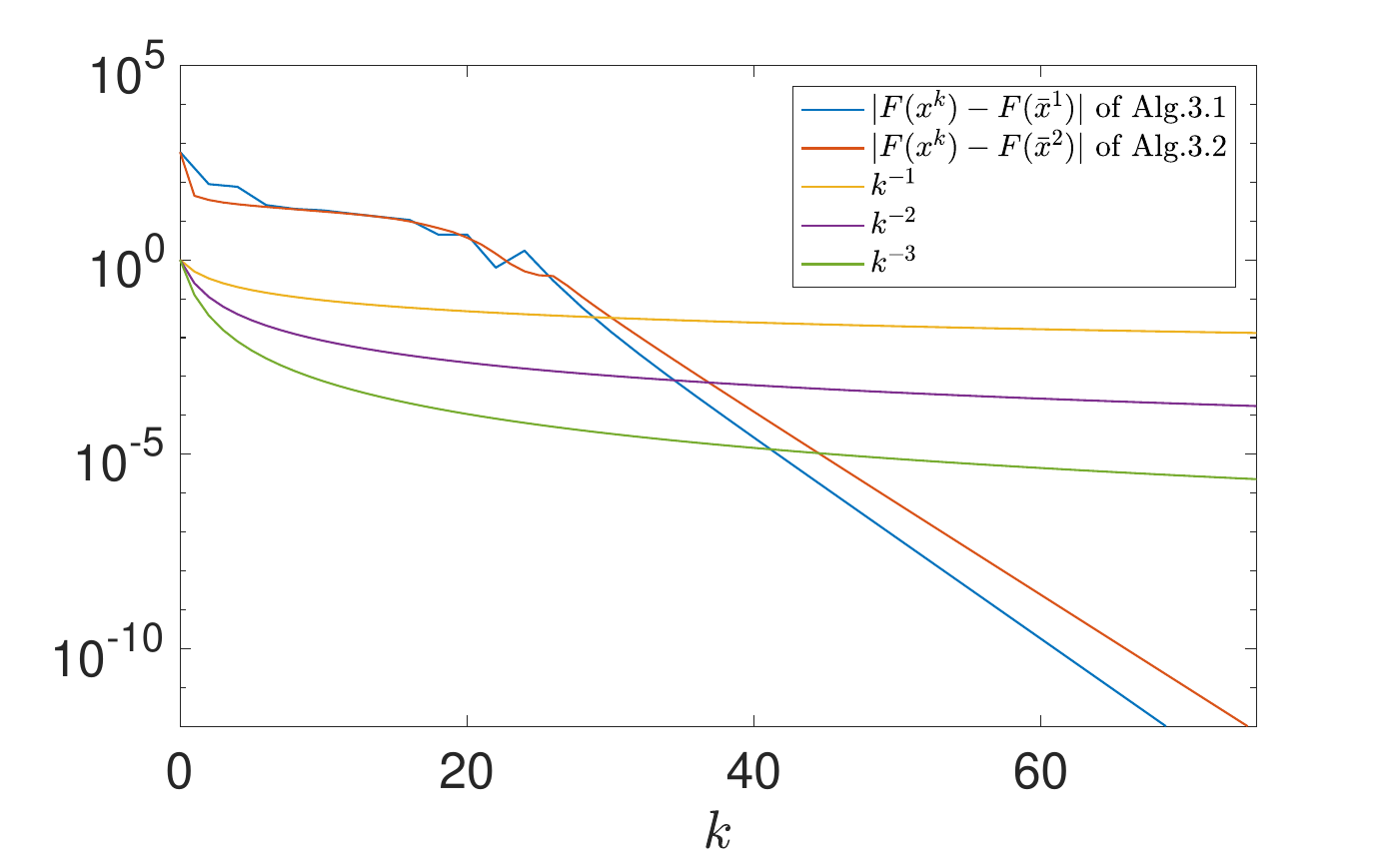}
		\label{fig_n160dF}}
	\hfil
	\subfigure[]{\includegraphics[width=2in]{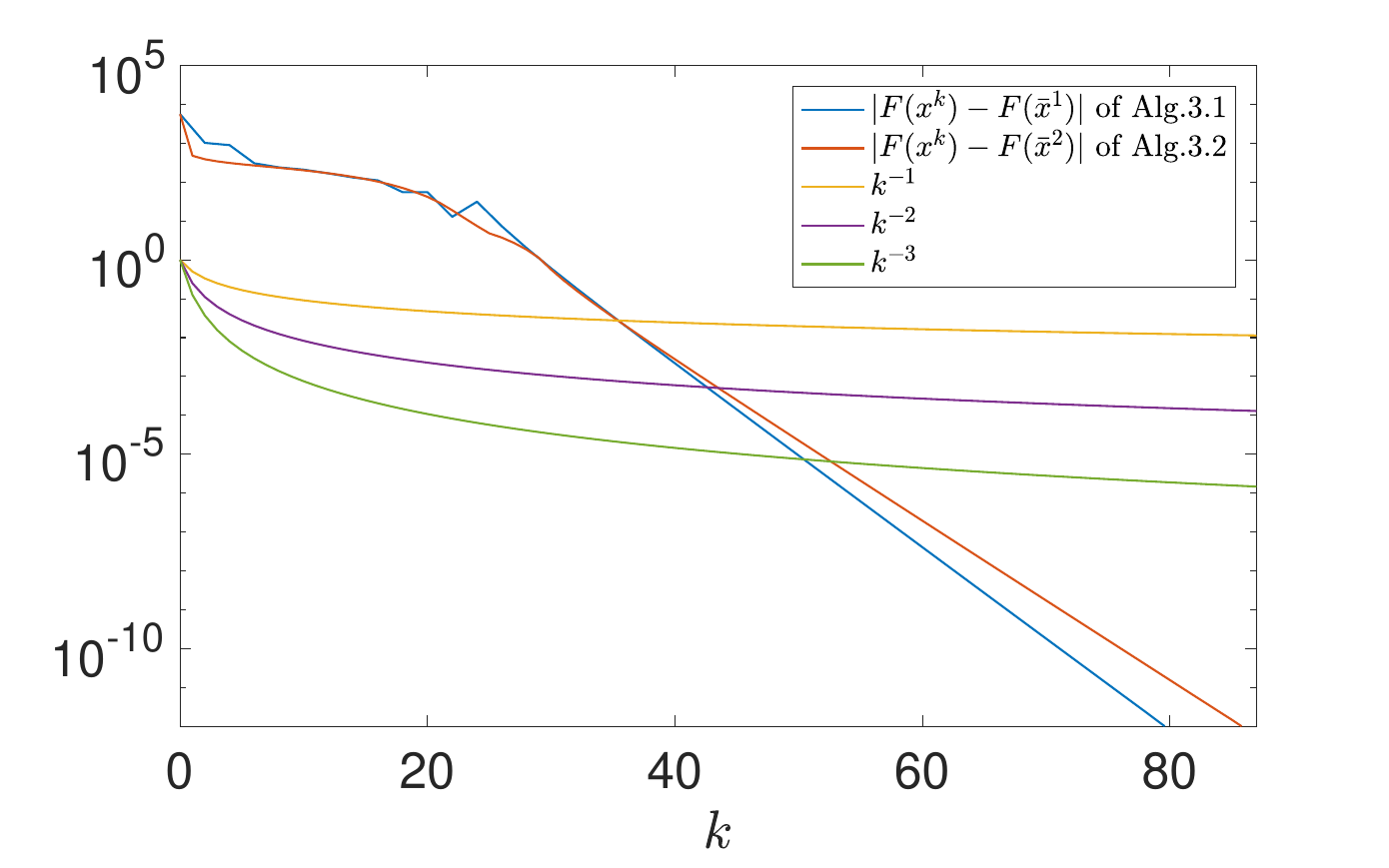}
		\label{fig_n1600dF}}
	\hfil
	\subfigure[]{\includegraphics[width=2in]{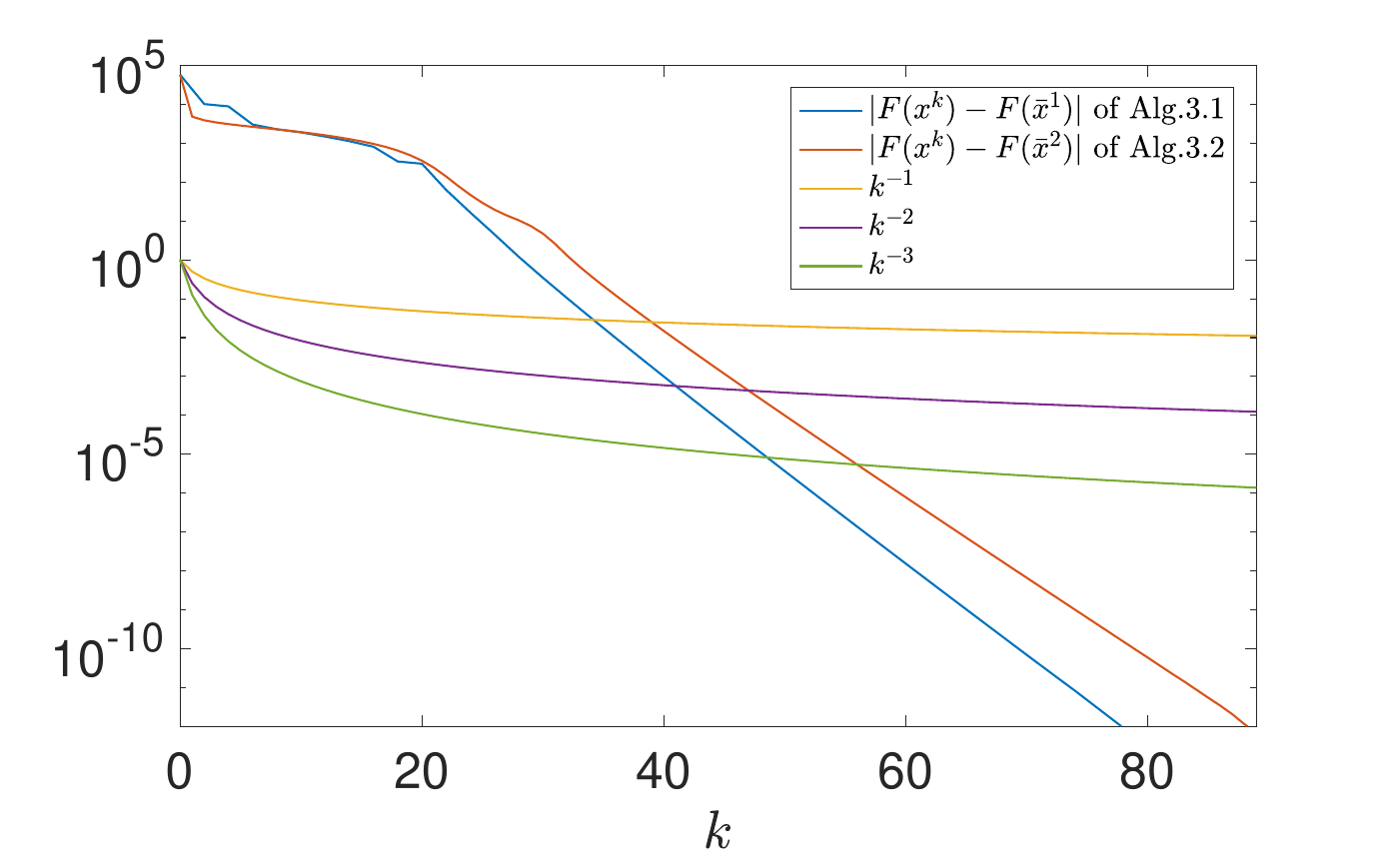}
		\label{fig_n16000dF}}
	\caption{$|F(x^{k})-F(\bar{x}^1)|$ by Alg.\ref{g0palg1}, $|F(x^{k})-F(\bar{x}^2)|$ by Alg.\ref{g0palg2}, $k^{-1}$, $k^{-2}$ and $k^{-3}$ (using logarithmic scale) against iteration $k$ in the cases of (a) $n=160$, (b) $n=1600$ and (c) $n=16000$.}\label{g0pdF}
\end{figure}
\begin{figure}
	\centering
	\subfigure[]{\includegraphics[width=2in]{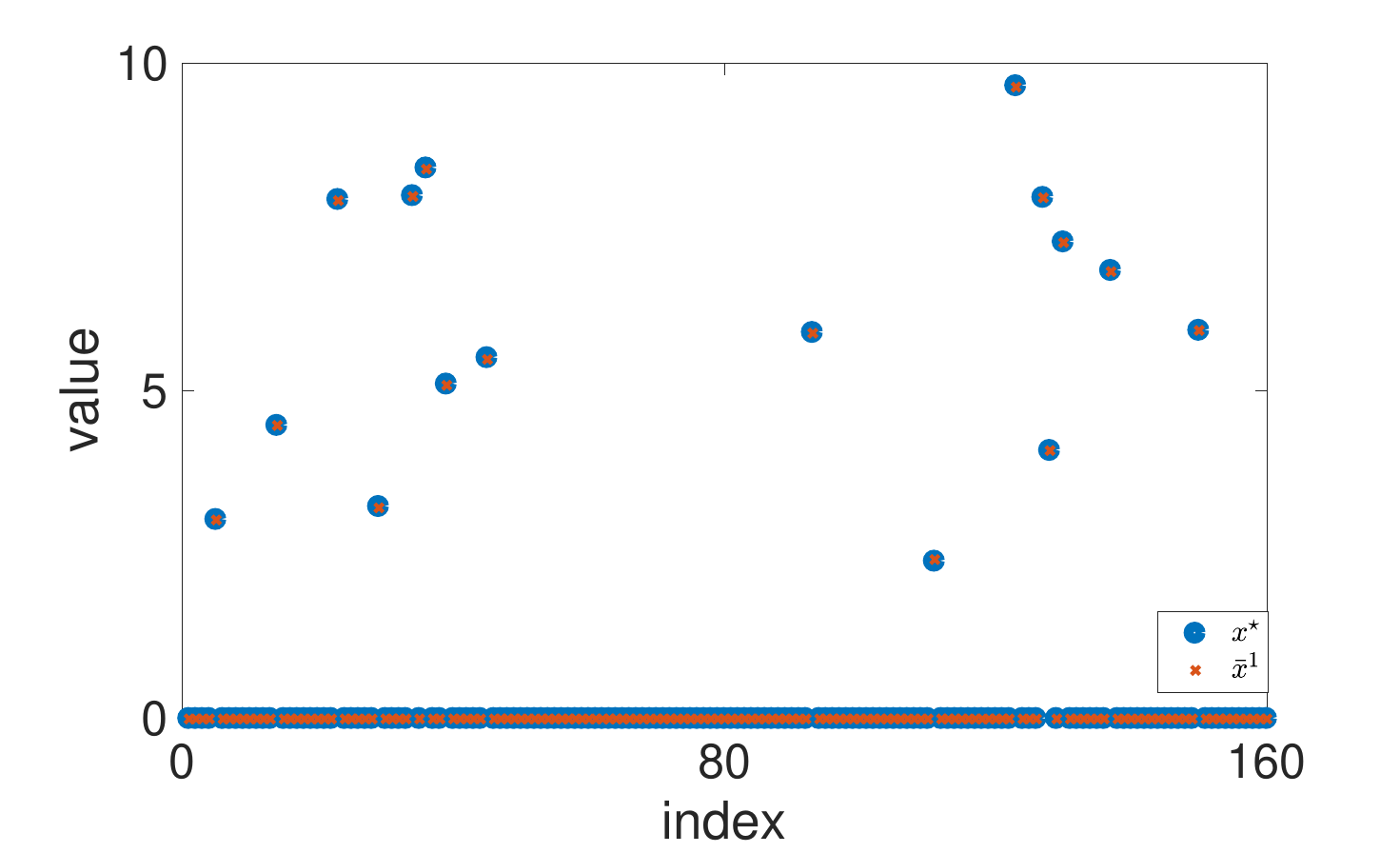}
		\label{fig_first_case}}
	\hfil
	\subfigure[]{\includegraphics[width=2in]{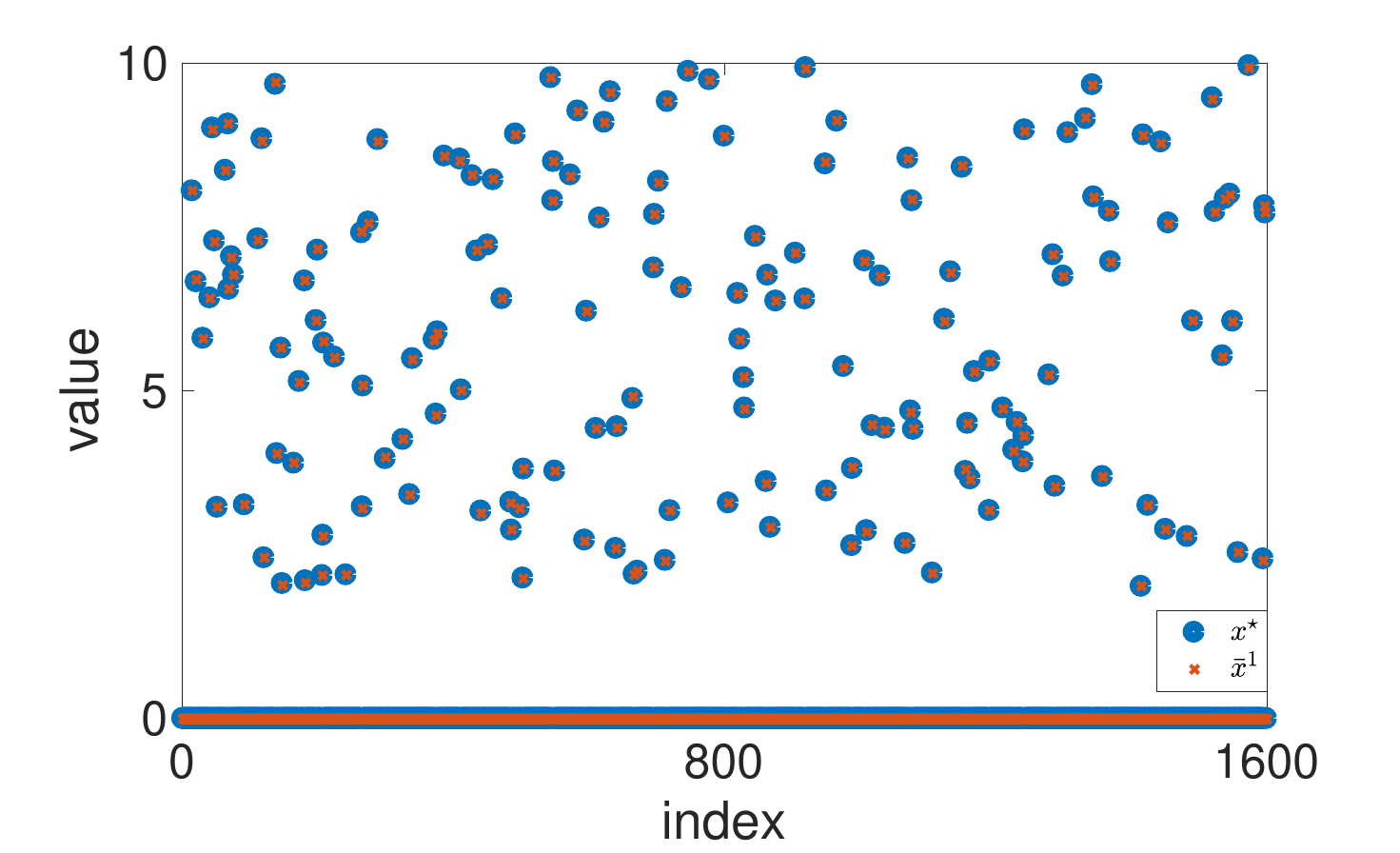}
		\label{fig_first_case2}}
	\hfil
	\subfigure[]{\includegraphics[width=2in]{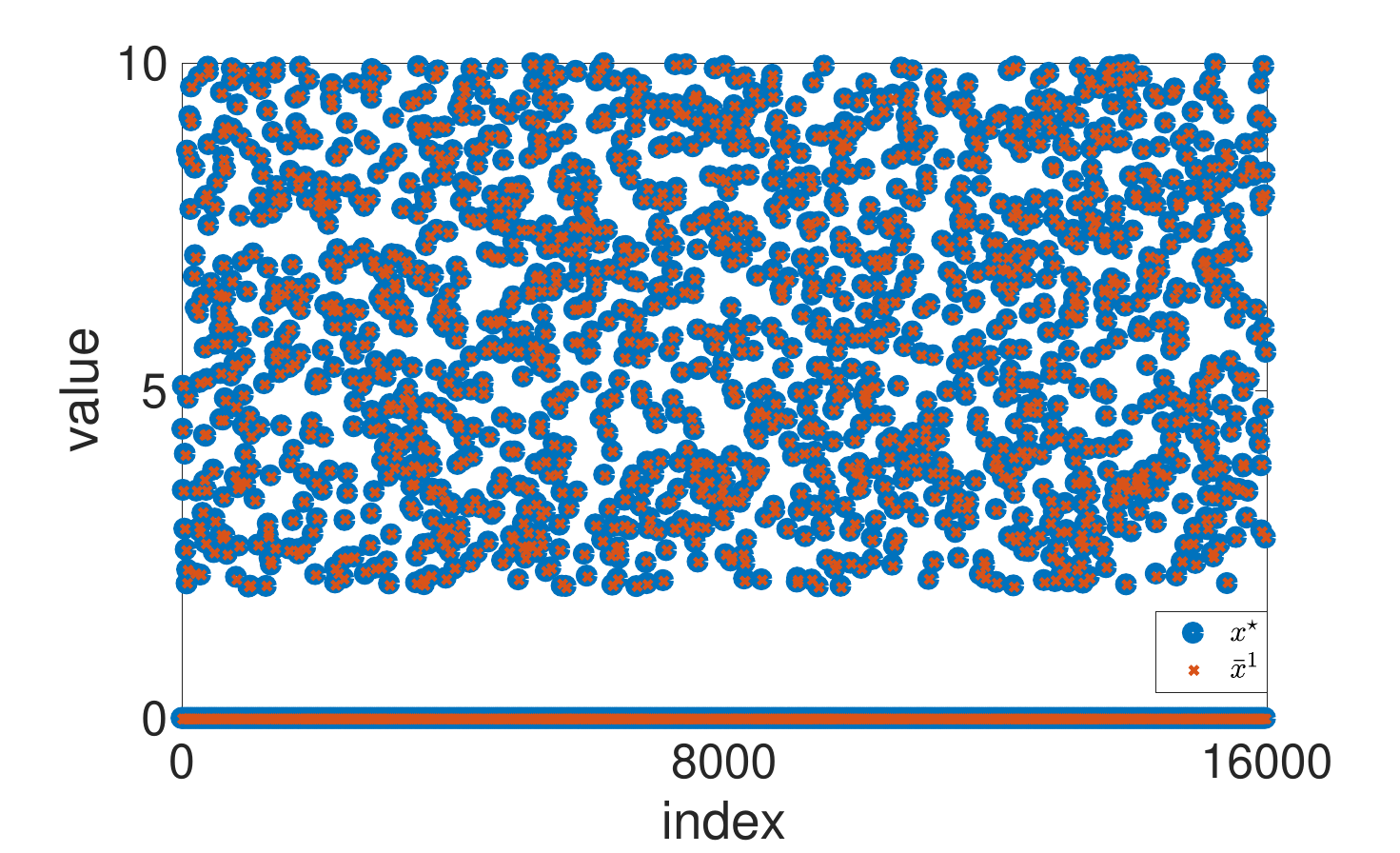}
		\label{fig_first_case3}}
	\hfil
	\subfigure[]{\includegraphics[width=2in]{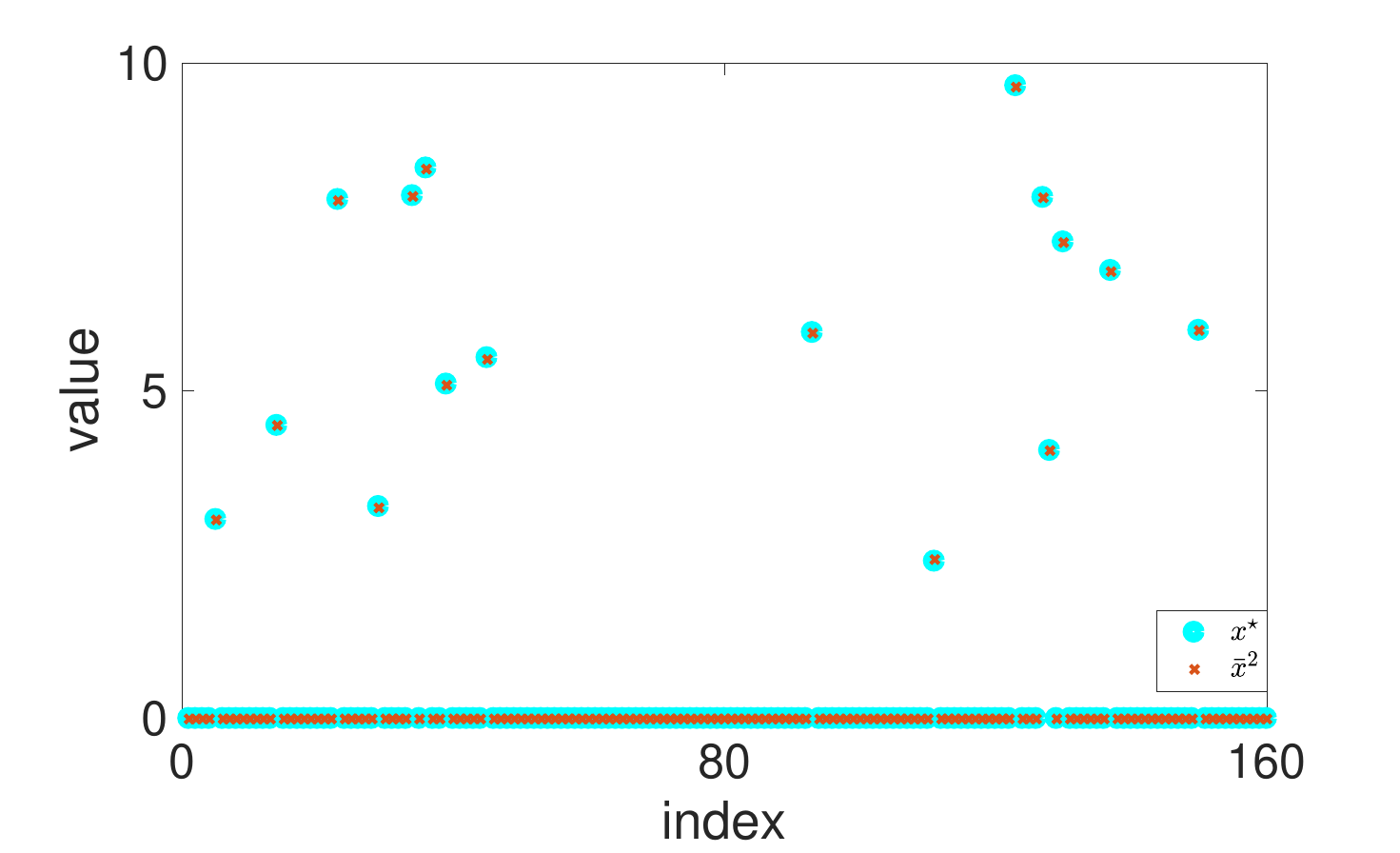}
		\label{fig_second_case}}	
	\hfil
	\subfigure[]{\includegraphics[width=2in]{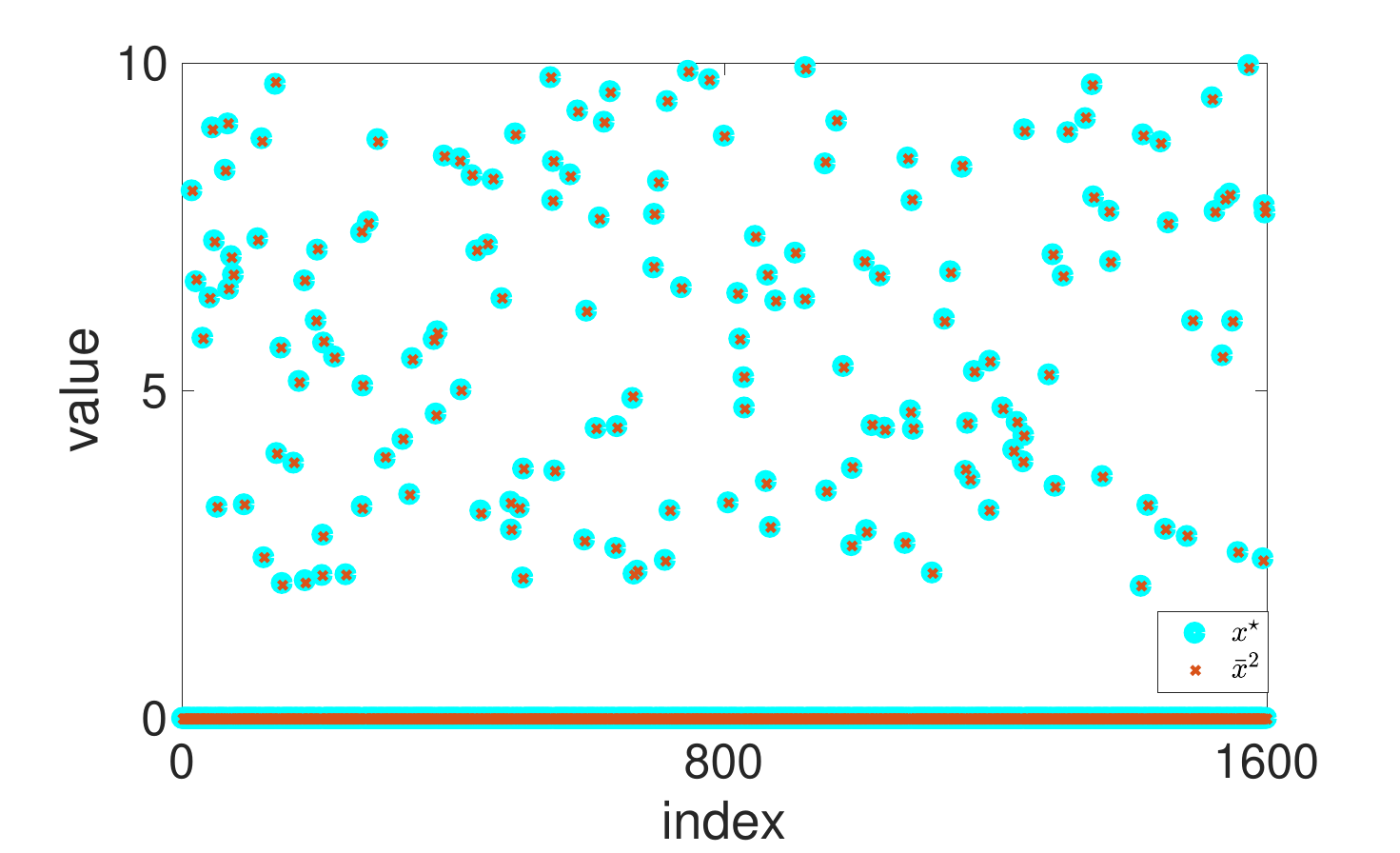}
		\label{fig_second_case2}}
	\hfil
	\subfigure[]{\includegraphics[width=2in]{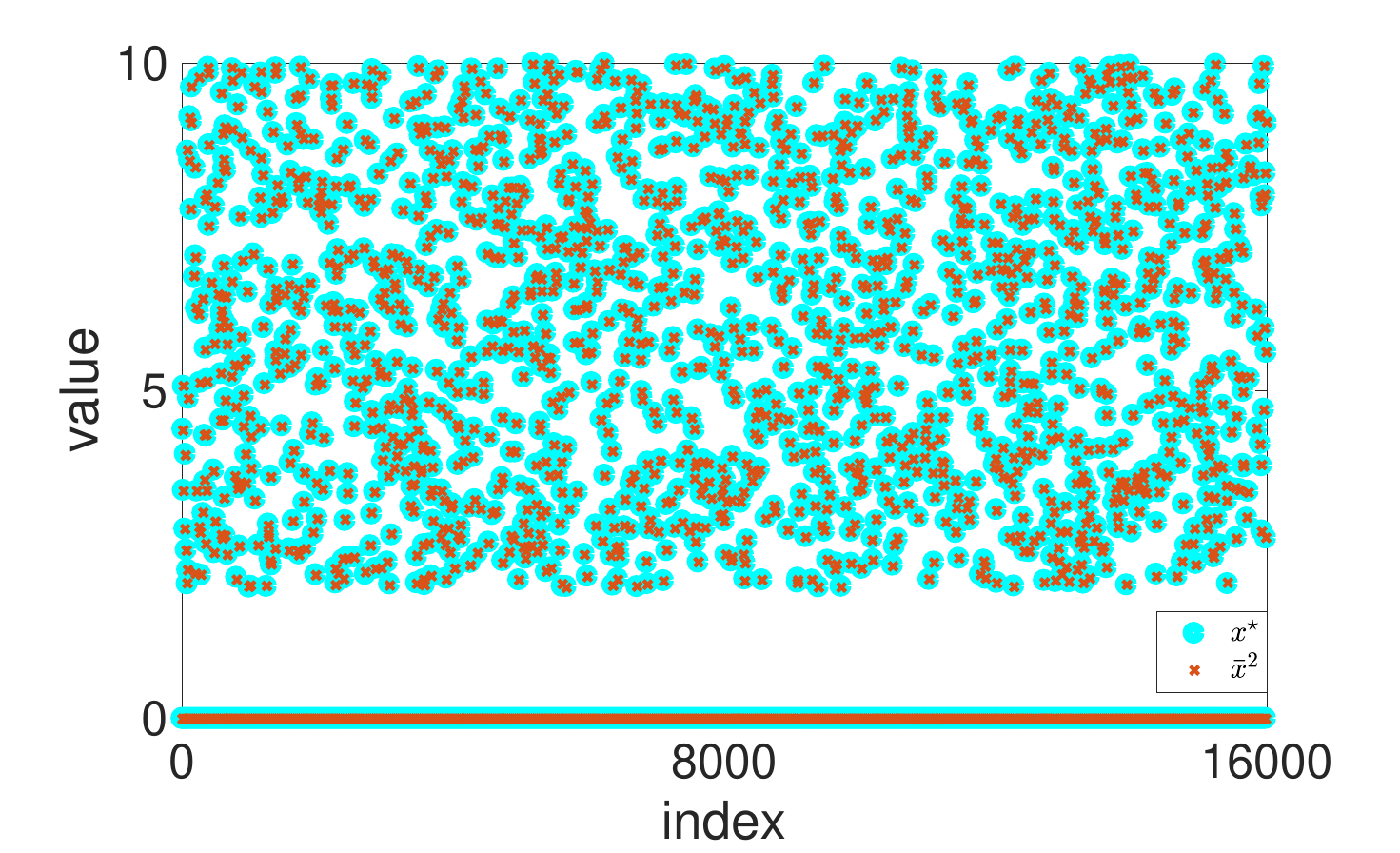}
		\label{fig_second_case3}}
	\caption{(a) (b) (c): Original signal $x^\star$ and restored signal $\bar{x}^1$ output by Alg.\ref{g0palg1}; (d) (e) (f): original signal $x^\star$ and restored signal $\bar{x}^2$ output by Alg.\ref{g0palg2} with $n=160,1600$ and $16000$, respectively.}\label{g0psls}
\end{figure}
\begin{figure}[h!]
	\centering
	\subfigure[]{\includegraphics[width=2in]{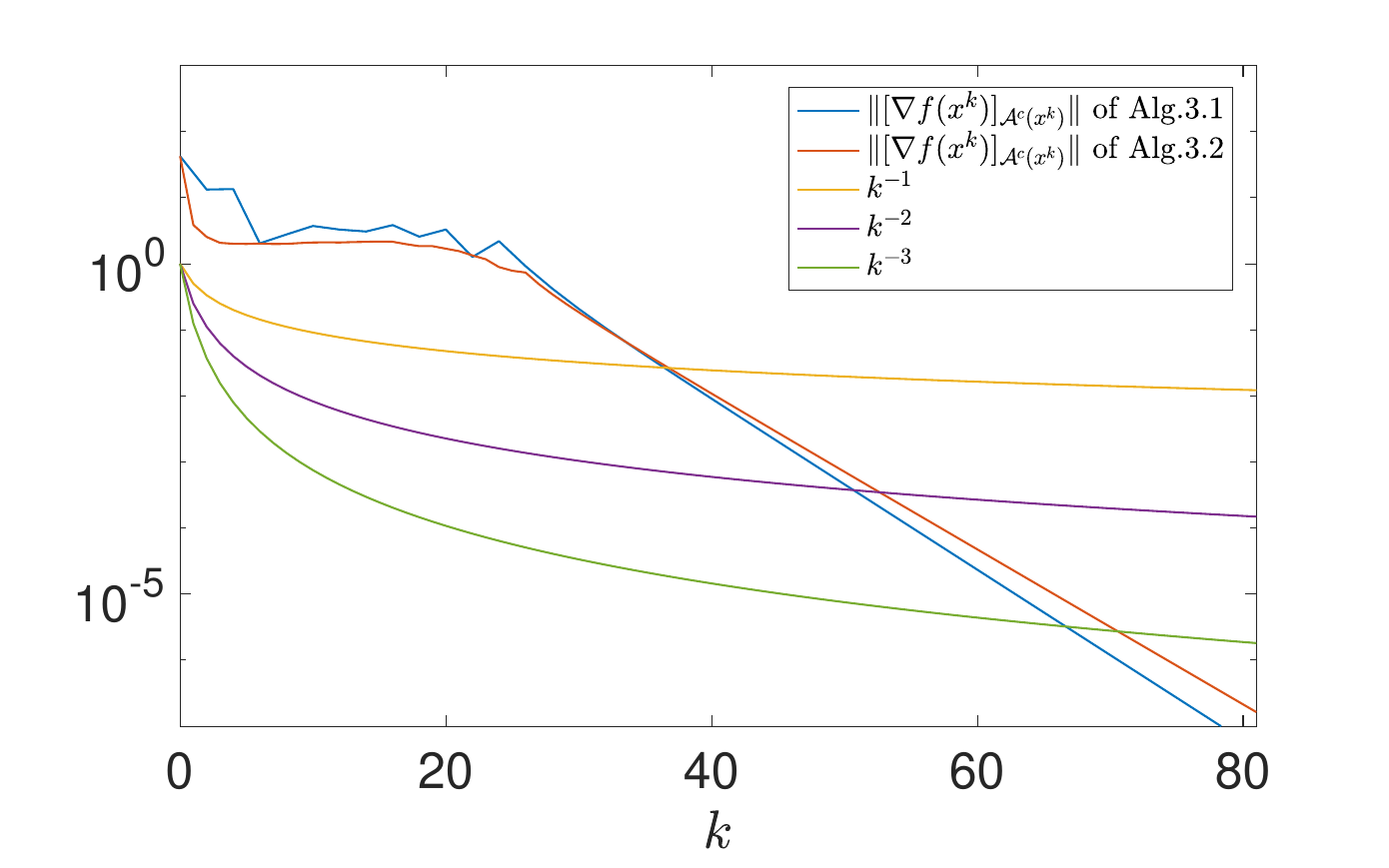}
		\label{fig_n160dFx}}
	\hfil
	\subfigure[]{\includegraphics[width=2in]{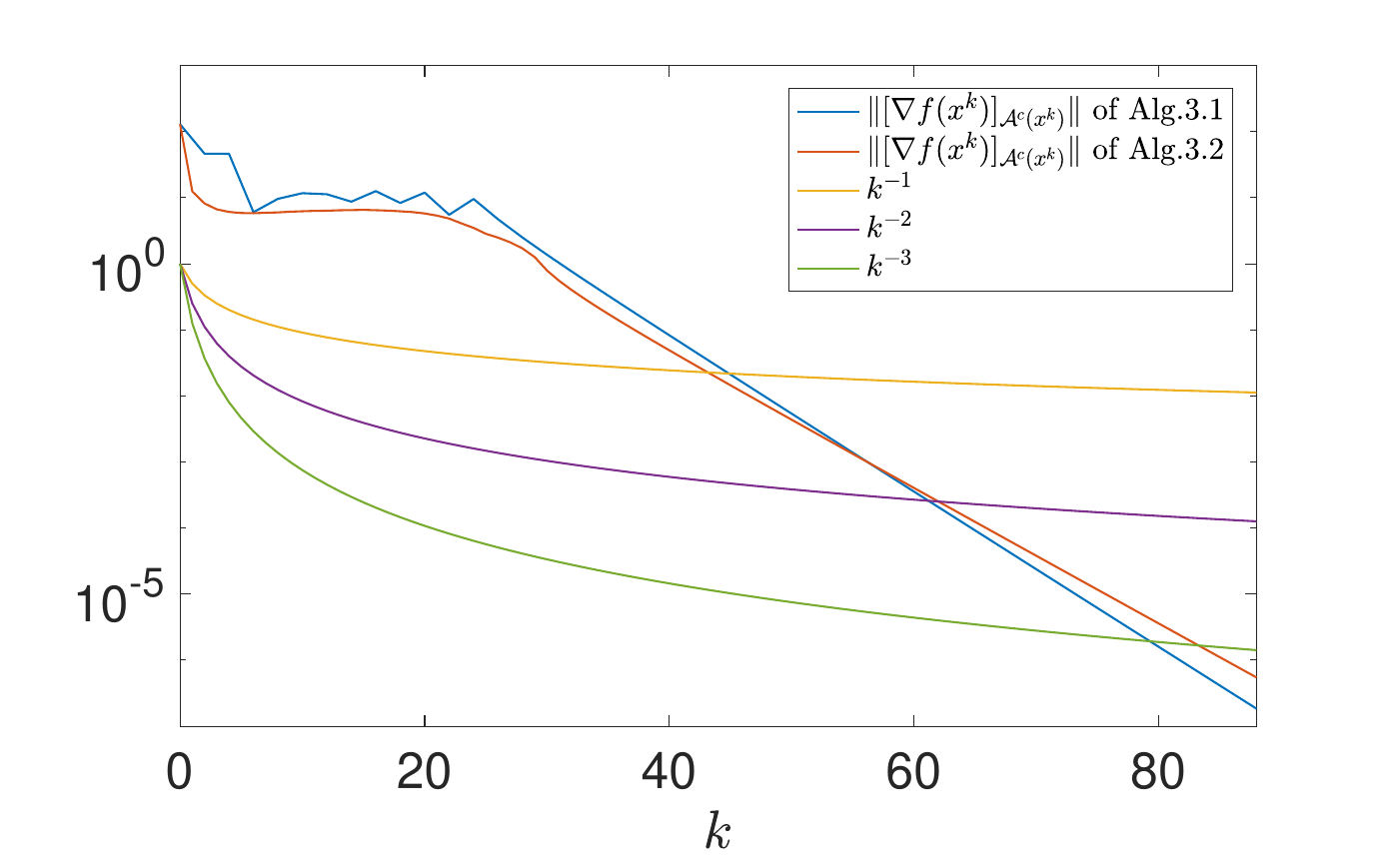}
		\label{fig_n1600dFx}}
	\hfil
	\subfigure[]{\includegraphics[width=2in]{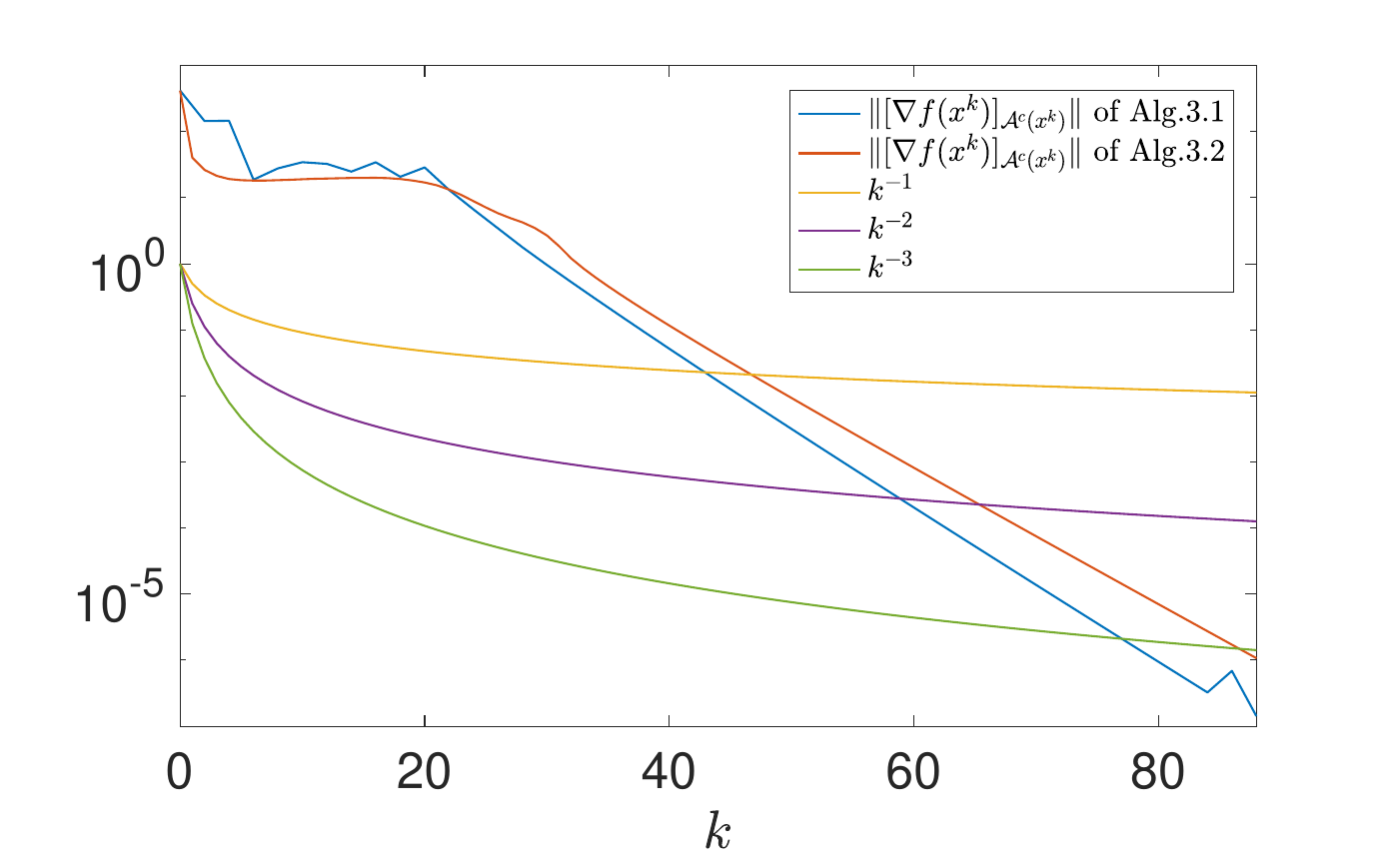}
		\label{fig_n16000dFx}}
	\caption{$\|[\nabla f(x^{k})]_{\mathcal{A}^c(x^{k})}\|$ by Alg.\ref{g0palg1} and Alg.\ref{g0palg2}, $k^{-1}$, $k^{-2}$ and $k^{-3}$ (using logarithmic scale) against iteration $k$ in the cases of (a) $n=160$, (b) $n=1600$ and (c) $n=16000$.}\label{g0pdFx}
\end{figure}

Considering the similar results of Alg.\ref{g0palg1} and Alg.\ref{g0palg2} shown in Table \ref{g0ptb1}, we will keep $n=1600,m=800,s=160$ and use $\alpha_k^B=L_s/8$, $L_s/15$ to further analyze the numerical performance of Alg.\ref{g0palg1} with different expressions of $\mu_k$, initial points, noise types and noise levels, respectively. In the following subsections, the parameters not mentioned are the same as those set above.

\subsubsection{Different expressions of $\mu_k$}\label{munes}
Set $\bar\mu_k=M-k/5$ with $M=0,4,20$ and $50$. The magnitude of $L_f$ is $10^4$. When $M=0$, we have $\bar\mu_k=-k/5$ and hence $\mu_k=\nu,\forall k\geq0$. Alg.\ref{g0palg1} takes $20,100,250$ outer loop iterations to reach $\bar{\mu}_k\leq\nu$ when $M=4,20,50$, respectively. From Table \ref{g0ptb2}, we can see that when $M=0$, Alg.\ref{g0palg1} with both $\alpha_k^B=L_s/8$ and $L_s/15$ fails in the signal recovery, but when $M=4,20$ and $50$, the algorithm has consistent good recovery results. Therefore, for the case with a small $\nu$, the finite dynamic update of $\mu_k$ is conducive for the iterates generated by Alg.\ref{g0palg1} with $\alpha_k^B=L_s/8$ and $L_s/15$ to converge to a $\nu$-strong local minimizer of (\ref{g0p}) that is close to the original signal. Moreover, we can see that the increase of $M$ for the values of $4, 20, 50$ leads to the growth in the number of outer loop iterations.
\begin{table}
	\renewcommand{\arraystretch}{1}
	\caption{MSEs and iterations of Alg.\ref{g0palg1} with $\alpha_k^B=L_s/8$ and $L_s/15$ for different $M$ in $\bar\mu_k$}
	\label{g0ptb2}
	\centering \begin{tabular}{|c||c||c|c|c|c|}
		\hline
		$\alpha_k^B$&$M$& $0$ & $4$ &  $20$ & $50$\\
		\hline\hline
		$L_s/8$&$k_1(\underline{k}_1)$ &$839.7(2.9)$& $44.5(2.9)$&  $123.5(2.9)$& $273.5(2.9)$\\
		\hline
		&MSE &$0.75$& $2.24\times10^{-5}$&  $2.24\times10^{-5}$& $2.24\times10^{-5}$\\
		\hline\hline
		$L_s/15$&$k_1(\underline{k}_1)$ &$159.6(3.9)$& $43.0(3.8)$&  $120.7(3.8)$& $270.7(3.7)$\\
		\hline
		&MSE &$1.11$& $2.24\times10^{-5}$&  $2.24\times10^{-5}$& $2.24\times10^{-5}$\\
		\hline
	\end{tabular}
\end{table}
\subsubsection{Different initialization}
Set initial point $x^0=\bm{-1},\bm{0},\bm{2}$ and a random point in $[-1,2]^n$. From Table \ref{g0ptb3}, we can see that Alg.\ref{g0palg1} with $\alpha_k^B=L_s/8$ and $L_s/15$ are insensitive to initializations, even the case that $x^0=\bm{-1}$, which is not in the feasible region. Moreover, by Table \ref{g0ptb2} and Table \ref{g0ptb3}, it can be seen that the inner loop of Alg.\ref{g0palg1} stops when $\alpha_k$ is close to $L_s/2$, which is consistent with Proposition \ref{innerloops}, and smaller than the lower bound of $L_s$ in the proof of Proposition \ref{innerloops}.
\begin{table}
	\renewcommand{\arraystretch}{1}
	\caption{MSEs and iterations of Alg.\ref{g0palg1} with $\alpha_k^B=L_s/8$ and $L_s/15$ for different initial points}
	\label{g0ptb3}
	\centering \begin{tabular}{|c||c||c|c|c|c|}
		\hline
		$\alpha_k^B$&$x^0$ & $\bm{0}$ &  $\bm{2}$ & Random& $\bm{-1}$\\
		\hline\hline
		$L_s/8$&$k_1(\underline{k}_1)$ &$40.1(2.9)$&  $42.1(2.9)$ & $41.6(2.9)$& $40.7(2.9)$\\
		\hline
		&MSE & $2.26\times10^{-5}$&  $2.26\times10^{-5}$& $2.26\times10^{-5}$&$2.26\times10^{-5}$\\
		\hline\hline
		$L_s/15$&$k_1(\underline{k}_1)$ &$37.9(3.8)$&  $41.3(3.9)$ & $59.1(3.8)$& $38.8(3.8)$\\
		\hline
		&MSE & $2.26\times10^{-5}$&  $2.26\times10^{-5}$& $2.26\times10^{-5}$&$2.26\times10^{-5}$\\
		\hline
	\end{tabular}
\end{table}
\subsubsection{Different noise types and noise levels}
In this part, we compare the numerical performance of Alg.\ref{g0palg1} with $\alpha_k^B=L_s/8$ and $L_s/15$ for problem (\ref{g0pep1}) under different noise types (Gaussian, Rayleigh, Gamma, Exponent, Uniform) and levels defined by $\sigma=0, 10^{-4}, 10^{-2}$. In generating $b=Ax^\star+\textsf{noise}$, the MATLAB codes of these noises are as follows.\\
{\tt{noise=$\sigma$*randn(m,1)(Gaussian);noise=$\sigma$*raylrnd(1,m,1)(Rayleigh);}}\\
{\tt{noise=$\sigma$*gamrnd(1,2,[m,1])(Gamma);noise=$\sigma$*exprnd(2,[m,1])(Exponent);}}\\
{\tt{noise=$\sigma$*unifrnd(0,2,[m,1])(Uniform);
}}

As can be seen from Table \ref{g0ptb6}, the less the noise level, the better the recovery effect of Alg.\ref{g0palg1}.
\begin{table}
	\renewcommand{\arraystretch}{1}
	\caption{MSEs of Alg.\ref{g0palg1} with $\alpha_k^B=L_s/8$ and $L_s/15$ for different noise types and noise levels}
	\label{g0ptb6}
	\centering \begin{tabular}{|c||c|c|c|c|c|}
		\hline
		$\alpha_k^B=\frac{L_s}{8}$ & Gaussian & Rayleigh & Gamma & Exponent & Uniform\\
		\hline
		$\sigma=0$ &$2.3\times10^{-18}$& $2.3\times10^{-18}$& $2.3\times10^{-18}$ & $2.3\times10^{-18}$& $2.3\times10^{-18}$\\
		\hline
		$\sigma=10^{-4}$ &$2.2\times10^{-9}$& $4.2\times10^{-9}$& $1.8\times10^{-8}$ & $1.8\times10^{-8}$& $2.9\times10^{-9}$\\
		\hline
		$\sigma=10^{-2}$ &$2.3\times10^{-5}$& $4.6\times10^{-5}$& $1.9\times10^{-4}$& $1.8\times10^{-4}$& $2.9\times10^{-5}$\\
		\hline\hline
		$\alpha_k^B=\frac{L_s}{15}$ &Gaussian & Rayleigh & Gamma & Exponent & Uniform\\
		\hline
		$\sigma=0$&$5.7\times10^{-18}$& $5.7\times10^{-18}$& $5.7\times10^{-18}$ & $5.7\times10^{-18}$& $5.7\times10^{-18}$ \\
		\hline
		$\sigma=10^{-4}$ &$2.2\times10^{-9}$& $4.2\times10^{-9}$& $1.8\times10^{-8}$ & $1.8\times10^{-8}$& $2.9\times10^{-9}$\\
		\hline
		$\sigma=10^{-2}$ &$2.3\times10^{-5}$& $4.6\times10^{-5}$& $1.9\times10^{-4}$& $1.8\times10^{-4}$& $2.9\times10^{-5}$\\
		\hline
	\end{tabular}
\end{table}
\subsection{Multichannel Image Reconstruction}\label{num-exp2}
In this experiment, we test the effectiveness of Alg.\ref{g0palg1} with different $N$ in recovering 2D images from compressive and noisy measurement \cite{Huang2009,Jiao2017}.

Consider the following group sparse $\ell_0$ regularized problem
\begin{equation}\label{g0pep2}
	\min_{x\in[-10,10]^n}~\|Ax-b\|^2+\lambda_1\|x\|_0+\lambda_2\sum_{l=1}^{n/3}\mathcal{I}(\|x_{(l)}\|_1).
\end{equation}

We adopt the same images and experimental setting of \cite{Jiao2017}. The original images have three channels. Each image is transformed into an $n$-dimensional vector $x^*$ and the pixels are grouped at the same position from three channels together. The observational data $b$ is generated by $b=Ax^*+\eta$, where all entries of $\eta$ follow an i.i.d. Gaussian distribution $N(0,\sigma^2)$. In the first image reconstruction, $A\in\mathbb{R}^{1152\times 6912}$ is a random Gaussian matrix, $\lambda_1=\lambda_2=10^{-1}$, $n=6912$ in (\ref{g0pep2}) and $\bar\mu_k=1-k/200$. In the second image reconstruction, $A\in\mathbb{R}^{3771\times 12288}$ is a composition of a partial FFT with an inverse wavelet transform and 6 levels of Daubechies wavelet, $\lambda_1=4\cdot10^{-4}$, $\lambda_2=4\cdot10^{-3}$, $n=12288$ in (\ref{g0pep2}) and $\bar\mu_k=1-k/1200$. We set the initial point $x^0=\bm{0}$ in Alg.\ref{g0palg1} for all cases.

We use group sparse $\ell_0$ regularized model (\ref{g0p}) and Alg.\ref{g0palg1} with $N=1,2,3$ to compare with four group sparse recovery models and algorithms, which are (i) least squares $\ell_{2,0}$ regularized model and GPDASC algorithm with the same continuation strategy along $\lambda$ in \cite{Jiao2017}; (ii) group OMP model and GOMP algorithm in \cite{Eldar2010}; (iii) group MCP model and GCD algorithm in \cite{Huang2012}; (iv) group Lasso model and SPGl1 algorithm in \cite{Van2008}. These four algorithms all rely on a reliable estimate of the noise level and the matrix $A$ needs to be normalized for the GPDASC algorithm. There is no regularization parameter in the group OMP, MCP and Lasso models. The MATLAB codes of SPGl1 algorithm are obtained from http://www.cs.ubc.ca/$\sim$mpf/spgl1/ and the MATLAB codes of the other algorithms, used in \cite{Jiao2017}, are obtained from http://www0.cs.ucl.ac.uk/staff/b.jin/software/gpdasc.zip. From Table \ref{g0ptb21}, it can be seen that Alg.\ref{g0palg1} with $N=1,2,3$ gives larger PSNR values than GPDASC, GOMP, GCD and SPGl1 for different noise levels, which implies that Alg.\ref{g0palg1} has better recovery performance. While GPDASC and SPGl1 give shorter runtimes, the PSNR values of their recovered images are clearly worse than those obtained by Alg.\ref{g0palg1}. Moreover, Alg.\ref{g0palg1} with $N=2$ has better recovery performance and less runtime than that with $N=1$ and $3$. The original images and restored images for the noisy case of $\sigma=0.02$ are presented in Fig.\ref{fig103} and Fig.\ref{fig02}.

\begin{table}
	\renewcommand{\arraystretch}{1}
	\caption{PSNRs and CPU time of Alg.\ref{g0palg1} with $N=1,2,3$, GCD, GPDASC, GOMP and SPGl1 for different noise levels}
	\label{g0ptb21}
	\centering \begin{tabular}{|c|c||c|c|c|c|c|}
		\hline
		$\sigma$&Test 1&Alg.\ref{g0palg1} with $N=1,2,3$& GCD & GPDASC &GOMP &SPGl1 \\
		\hline
		0.01&PSNR&$31.3~~~~31.3~~~~31.3$& $28.5$ & $28.5$& $28.5$& $22.9$\\
		&Time(s)&$1.3~~~~~1.3~~~~~1.3$& $21.0$ & $3.6$& $6.4$& $0.4$\\
		\hline
		0.02&PSNR&$31.2~~~~31.2~~~~31.2$& $22.5$ & $22.5$& $22.4$& $22.8$\\
		&Time(s)&$1.3~~~~~1.3~~~~~1.3$& $21.5$ & $3.7$& $6.7$& $0.6$\\
		\hline
		0.03&PSNR &$31.3~~~~31.3~~~~31.3$& $19.0$ & $19.0$& $18.5$& $22.5$\\
		&Time(s) &$1.1~~~~~1.1~~~~~1.1$& $20.9$ & $3.9$& $6.8$& $0.4$\\
		\hline\hline
		$\sigma$&Test 2&Alg.\ref{g0palg1} with $N=1,2,3$& GCD & GPDASC &GOMP &SPGl1 \\
		\hline
		0.01&PSNR&$42.8~~~~43.2~~~~42.9$& $42.6$& $42.1$ & $39.8$& $24.7$\\
		&Time(s)&$76.1~~~~75.3~~~~76.5$& $182.1$ & $41.7$& $168.0$& $5.1$\\
		\hline
		0.02&PSNR&$36.8~~~~37.1~~~~36.7$& $36.1$& $33.8$ & $33.7$& $23.7$\\
		&Time(s)&$77.9~~~~77.0~~~~77.8$& $160.6$ & $35.7$& $158.8$& $4.5$\\
		\hline
		0.03&PSNR &$32.0~~~~32.3~~~~32.2$& $31.9$ & $30.7$& $30.1$& $22.6$\\
		&Time(s) &$77.2~~~~76.0~~~~77.2$& $149.1$ & $36.0$& $145.1$ & $3.0$\\
		\hline
	\end{tabular}
\end{table}
\begin{figure}
	\centering
	\includegraphics[width=6.5in]{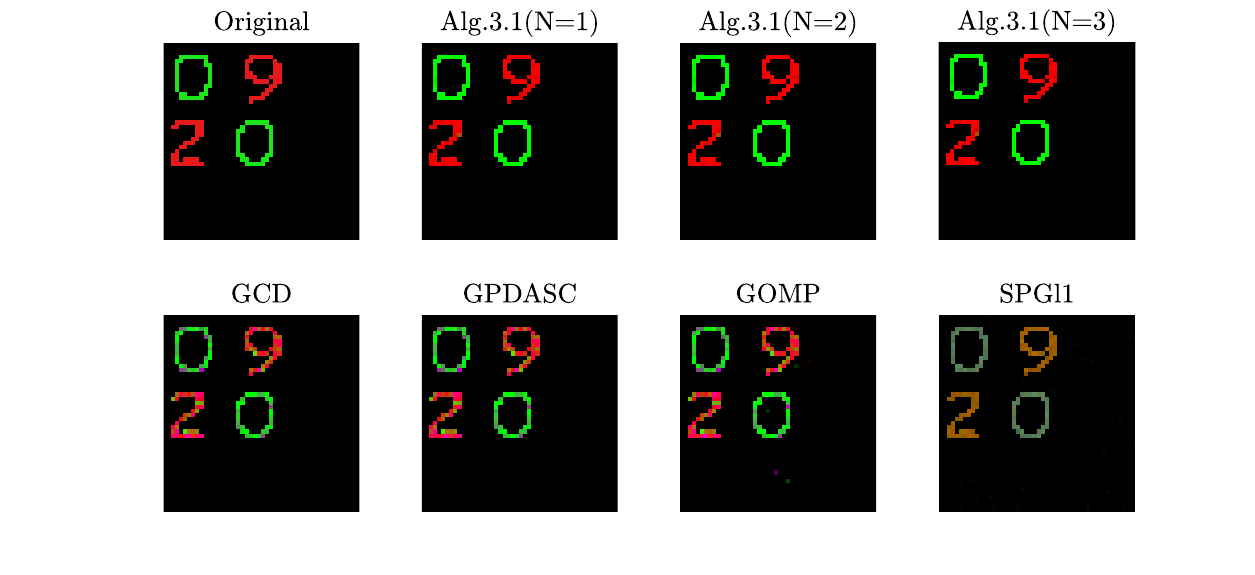}
	\caption{The first original image and restored images by Alg.\ref{g0palg1} with $N=1,2,3$, GCD, GPDASC, GOMP and SPGl1 when $\sigma=0.02$}\label{fig103}
\end{figure}
\begin{figure}
	\centering
	\includegraphics[width=6.5in]{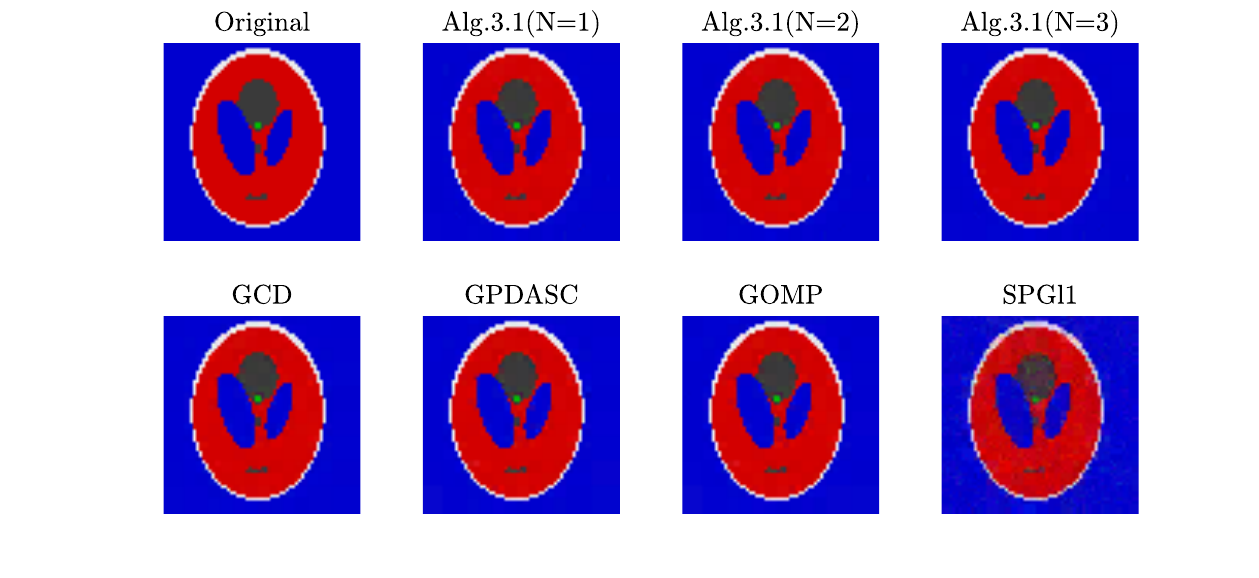}
	\caption{The second original image and restored images by Alg.\ref{g0palg1} with $N=1,2,3$, GCD, GPDASC, GOMP and SPGl1 when $\sigma=0.02$}\label{fig02}
\end{figure}
\section {Conclusions}\vspace{0.5ex}
In this paper, we designed two kinds of DC algorithms to solve a class of sparse group $\ell_0$ optimization problems modeled by (\ref{g0p}). We gave the relaxation model (\ref{g0ps}) of problem (\ref{g0p}) and proved their equivalence. Based on the DC structure of the relaxation problem, some of the existing DC algorithms can solve it but only get its critical point. We defined the sw-d-stationary points of the relaxation problem (\ref{g0ps}), which have stronger optimality conditions than its critical points and weak d-stationary points, and designed two DC algorithms with convergence to the defined sw-d-stationary points. Comparing with the existing DC algorithms for the sparse group optimization problems, we established the relationship between the obtained solution by the proposed DC algorithms and the considered problem (\ref{g0p}). We proved that any accumulation point of the iterates generated by the proposed algorithms is a local minimizer of problem (\ref{g0p}) and satisfies a lower bound property of its global minimizers. Moreover, we proved that all accumulation points have a common support set, their nonzero entries have a common lower bound and their zero entries can be attained within finite iterations. Further, we proved the global convergence and fast convergence rate of the proposed algorithms under the mild condition. Finally, we illustrated the theoretical results and showed the good performance of the proposed DC algorithms by some numerical experiments.

\section*{Acknowledgments}
The authors are grateful to the associate editor and the two anonymous referees for their comments and suggestions that substantially improved the quality of the paper.

\bibliographystyle{plain}
\bibliography{references}
\end{document}